\documentclass[12pt]{amsart}
\usepackage{tikz}
\usepackage{amsmath,amsfonts,amsthm,amssymb}
\numberwithin{equation}{section}

\begin{document}
 
\newcommand\N{{\mathbb N}}
\newcommand\R{{\mathbb R}}
\newcommand\Z{{\mathbb Z}}
\newcommand\C{{\mathbb C}}
 \newtheorem{theorem}{Theorem}
 \newtheorem{remark}{Remark}[section]
 \newtheorem{lemma}[remark]{Lemma}
 \newtheorem{cor}[remark]{Corollary}
 \newtheorem{proposition}[remark]{Proposition}
 \newtheorem{definition}[remark]{Definition}
\newcommand\etf{{e^{{\tau} \phi}}}
\newcommand\tnf{\tau \nabla \phi}
\def\f12{{\frac{1}{2}}}
\newcommand\ty{\tilde{y}}
\newcommand\rar{\rightarrow}
\newcommand\e{\varepsilon}
\newcommand\g{\gamma}
\newcommand\rt{\tilde{r}}
\newcommand\lip{\text{Lip}}
\newcommand\id{I_n}
\newcommand\supp{{ \text{supp }}}
\newcommand\jy{{\langle y \rangle}}
\newcommand\tsgn{\chi}
\newcommand\bb{\beta}
\newcommand\dd{\delta}
\def\d{\partial}
\def\aa{\alpha}
\def\l{\lambda}
\def\th{\tilde{H}}
\def\eSUCP{{\bf ESUCP}\ }
\def\SUCPv{{\bf SUCP(II)}\ }
\def\SUCP{{\bf SUCP(I)}\ }
\def\UCP{{\bf UCP}\ } 
\def\nt{\tilde{\nabla}}
\def\pt{\tilde{P}}
\def\td{\tilde{\d}}
\def\tX{\tilde{X}}
\def\tY{\tilde{Y}}
\def\dt{P_0}
\def\Pt{\tilde{P}}
\def\Ptpsi{\tilde{P}_{\psi}}
\newcommand\dtpsio{{P}_{0,\psi}}
\newcommand\dtpsi{{P}_{\psi}}
\renewcommand\div{{\text{div }}}
\def\oops{{\bf !!!!!}}

\title{Carleman estimates and unique continuation for second order
parabolic equations with nonsmooth coefficients}

\author{Herbert Koch} 

\address{Mathematisches Institut der Universit\"at Bonn, Beringstr.1,
  53115 Bonn, Germany}

\email{koch@math.uni-bonn.de} 

\author{ Daniel Tataru}

\address{Department of Mathematics, University of California,
  Berkeley, CA 94720} \email{tataru@math.berkeley.edu} 

\thanks{The first author was supported in part by the DFG grant
  KO1307/5-3.  The second author was supported in part by the NSF
  grant DMS-0301122. Part of the work was done while the first author
  was supported by the Miller Institute for Basic Research in
  Science.}

\begin{abstract}
  In this work we obtain strong unique continuation results for
  variable coefficient second order parabolic equations. The
  coefficients in the principal part are assumed to satisfy a
  Lipschitz condition in $x$ and a H\"older $C^{\frac13}$ condition in
  time.  The coefficients in the lower order terms, i.e. the potential
  and the gradient potential, are allowed to be unbounded and required
  only to satisfy mixed norm bounds in scale invariant $L^p_t L^q_x$
  spaces.
\end{abstract}

\maketitle
\section{Introduction}

The evolution of the understanding of the strong unique
continuation problem for second order parabolic equations 
mirrors and is closely related to the corresponding strong unique
continuation problem for second order elliptic equations.
Consequently, we begin with a brief overview of the latter
problem.

To a second order elliptic operator $\Delta_g = \partial_i g^{ij}
\partial_j$ and potentials $V$, $W$ in $\R^n$ we associate the
elliptic equation
\begin{equation}
  \label{e}
  -\Delta_g u = W \nabla u + V u
\end{equation}
Given a function $u \in L^2_{loc}(\R^n)$ and $x_0 \in \R^n$
we say that $u$ vanishes of infinite order at $x_0$ if there
exists $R$ so that for each integer $N$ we have
\begin{equation}\label{ewvanishing}
\int_{B(x_0,r)} |u|^2 \, dx \le c_N^2 r^{2N}, \ \ \
  \ \ \ r < R
\end{equation}
The elliptic strong unique continuation property \eSUCP has the form
\[
\parbox[c]{4.1in}{\it Let $u$ be a solution to \eqref{e} which vanishes
  of infinite order at $x_0$. Then $u(x)=0$ for $x$ in a
  neighborhood of $x_0$.} \hspace{.5in} \eSUCP
\]

\eSUCP type results go back to the pioneering work of Carleman \cite{MR0000334}
in dimension $n=2$, later extended to higher dimension in by Aronszajn
and collaborators \cite{MR0092067}, \cite{MR0140031}. Their results apply to Lipschitz
metrics $g$ but only mildly unbounded potentials $V$ and $W$.
A key ingredient in their approach was to obtain a class of weighted
$L^2$ estimates which were later called Carleman estimates. The
simplest Carleman estimate has the form
\[
\| |x|^{-\tau} u\|_{L^2} \lesssim \| |x|^{2-\tau} \Delta u\|_{L^2}
\]
and holds uniformly for $\tau$ away from $\pm(\frac{n-2}{2} + N)$.
This restriction is related to the spectrum of the spherical
Laplacian.

Adding some extra convexity to the $|x|^{-\tau}$ weight makes the
above estimate more robust and allows one to also use it in the
variable coefficient case. The role played by the convexity was
further clarified and explained by H\"ormander \cite{MR686819}, \cite{MR1481433}, who
introduced the {\em pseudoconvexity} condition for weights as an
almost necessary and sufficient condition in order for the Carleman
estimates to hold.

The problem becomes more difficult if one seeks to work with unbounded
potentials $V$ in or near the scale invariant $L^\frac{n}2$ space.
There the $L^2$ Carleman estimates are insufficient. Instead the key
breakthrough was achieved in Jerison-Kenig~\cite{MR87a:35058}, where
the $L^2$ Carleman estimates are replaced by $L^p$ estimates of the
form
\[
\| |x|^{-\tau} u\|_{L^\frac{2n}{n-2}} \lesssim \| |x|^{-\tau} \Delta u\|_{L^\frac{2n}{n+2}}
\] 
Relevant to the present paper is also the alternative  proof of this
result which was given by Jerison~\cite{MR865834}, taking advantage of Sogge's
\cite{MR930395} spectral projection bounds for the spherical Laplacian.  In
the case of operators with smooth variable coefficients $L^p$ Carleman
estimates were first obtained by Sogge~\cite{MR91d:35037},
\cite{MR91k:35068}.

Working with gradient potentials in the scale invariant space $W \in L^n$ 
introduces an added layer of difficulty. There not even the $L^p$
Carleman estimates can hold. Wolff's solution to this in
\cite{780.35015} is a weight osculation argument, which allows one to
taylor the weight in the Carleman estimate to the solution $u$,
producing estimates of the form
\[
\| e^{-\tau \phi(x)} u\|_{L^\frac{2n}{n-2}} 
+ \| e^{-\tau \phi(x)} W \nabla u\|_{L^\frac{2n}{n+2}} 
\lesssim \| e^{\tau \phi(x)} \Delta u\|_{L^\frac{2n}{n+2}}
\] 
where the choice of $\phi$ depends on both $u$, $W$ and $\tau$.
 
Finally, the authors's article \cite{MR2001m:35075} combines the ideas
above into a nearly optimal scale invariant \eSUCP result for the
elliptic problem, with (i) a Lipschitz metric $g$, (ii) an
$L^\frac{n}2$ potential $V$, and (iii) an almost $L^n$ gradient
potential $W$. The present paper is the counterpart of
\cite{MR2001m:35075} for the parabolic strong unique continuation
problem.

We consider the second order backwards parabolic operator
\begin{equation}
P = \d_t+ \d_k g^{kl}(t,x) \d_l
\label{p}\end{equation}
in $\R \times \R^n$ and potentials $V,W_1,W_2$. To these we associate the
parabolic equation
\begin{equation}
  \label{eq} 
  Pu =  Vu+W_1 \nabla_x u + \nabla_x (W_2 u) 
\end{equation}
Given a function $u \in L^2_{loc}$ and $(t_0,x_0) \in \R \times \R^n$
we say that $u$ vanishes of infinite order at $(t_0,x_0)$ if there
exists $R$ so that for each integer $N$ we have
\begin{equation}\label{wvanishing}
  \int_0^{r^2} \int_{B(x_0,r)} |u|^2 \, dx dt \le c_N^2 r^{2N}, \ \ \ \ \ \ r < R\end{equation}
Alternatively we may only require that $x \to u(t_0,x)$ vanishes of infinite 
order at $x_0$, i.e.
\begin{equation}\label{vanishing}
  \int_{B(x_0,r)} |u(t_0,x)|^2 \, dx  \le c_N^2 r^{2N}, \ \ \ \ \ \ r < R
\end{equation}
The two conditions \eqref{wvanishing} and \eqref{vanishing} are
largely equivalent provided that the coefficients $g^{kl}$ have some 
uniform regularity as $t \to t_0$. However, our assumptions in this
article are not strong enough to guarantee this, therefore we consider
the two separate cases.

Now we can define the strong unique continuation property \SUCP:
\[
\parbox[c]{4.2in}{\it Let $u$ be a solution to (\ref{eq}) which vanishes
  of infinite order at $(t_0,x_0)$. Then $u(t_0,x)=0$ for $x$ in a
  neighborhood of $x_0$.} \hspace{.4in} \SUCP
\]
and the slightly stronger variant
\[
\parbox[c]{4.6in}{\it Let $u$ be a solution to (\ref{eq}) so that $x\!
  \to \!u(t_0,x)$ vanishes of in\-fi\-nite order at $x_0$. Then $u(t_0,x)=0$
  for $x$ in a neighborhood of $x_0$.} \hspace{.1in} \SUCPv
\]

The study of unique continuation for parabolic equations 
began with early work of Mizohata~\cite{MR0106347} and
Yamabe~\cite{MR0101395}, followed by 
Saut-Scheurer~\cite{MR871574}; $L^p$ Carleman estimates 
were first obtained by Sogge~\cite{MR91m:35051}.

The study of the parabolic strong unique continuation problem began
with work of Lin~\cite{MR1024191} who considered \SUCPv for the heat
equation with $W=0$ and $V$ bounded and time independent. 
This continued with work of Chen~\cite{MR1637972} and Poon
\cite{MR1387458}.  Fern\'andez \cite{MR2001174}, and 
Escauriaza, Fern\'andez and
Vessella \cite{MR2198840} considered \SUCPv under various assumptions
on the coefficients and pointwise bounds for $W=0$ and $V$.
It is a consequence of Alessandrini and Vessella \cite{MR946281} that \SUCP and\SUCPv are equivalent under weak assumptions on the coefficients, and
they derived \SUCPv in \cite{MR2022375} for bounded $W$ and $V$.

The article of Poon \cite{MR1387458} contributed to clarifying 
the correct form of the $L^2$ Carleman estimates for
the parabolic strong unique continuation problem in  
Escauriaza and Fern\'andez's work  \cite{MR1971939}. In the simplest
form, these have the form
\[
\| t^{-\tau-\frac12} e^{-\frac{x^2}{8t}} u\|_{L^2}
\lesssim \| t^{-\tau+\frac12} e^{-\frac{x^2}{8t}} (\d_t + \Delta)  u\|_{L^2}
\]
and hold uniformly with respect to $\tau$ away from $(2n+\N)/4$.
This restriction is connected with the spectral properties 
of the Hermite operator.

The $L^p$ spectral projection bounds for the Hermite operator were
independently obtained by Thangavelu~\cite{MR958904} and
Kharazdhov~\cite{MR1319486}; see also the simplified proof in the
authors's paper \cite{MR2140267}. These bounds were essential in the
proof of $L^p$ Carleman inequalities for the heat operator of
Escauriaza \cite{MR1769727} and Escauriaza and Vega \cite{MR1871351}
which yield \SUCP when $g=I_n$, $W = 0$ and $V \in L^1L^{\infty} +
L^\infty L^{n/2}$.

Our aim is to prove that \SUCP respectively, \SUCPv hold under sharp
scale invariant assumptions on the metric $g$ and $L^p$ conditions on
the potentials $V$ and $W_1, W_2$.  The contribution of this work is
comparable to \cite{MR2001m:35075} for the elliptic problem: We study
almost optimal conditions on
\begin{enumerate}
\item the coefficients $g$
\item the potential $V$
\item the gradient potentials $W_j$
\end{enumerate}
The combination of rough variable coefficients and $L^p$
conditions on the potential seems to be new. Also, to the best of our
knowledge this is the first result on unique continuation for
parabolic problems under $L^p$ conditions on the coefficients of the
gradient term.
 
For simplicity we always assume that $t_0=0, \ x_0=0$. For \SUCP it is
natural\footnote{This becomes clearer later on after a change of
  coordinates and conjugation with respect to a Gaussian weight} to
consider a larger class of operators $P$ which have the form
\begin{equation}
P = \frac{\d}{\d t} + \d_k g^{kl} \d_l + \frac{x_k}{t} d^{kl} \d_l
+  \d_l d^{kl} \frac{x_k}{t} + \frac{x_k}{t} e^{kl}\frac{x_l}{t}
\label{pgen}\end{equation}
where $(g^{kl})$, $(d^{kl})$ and  $(e^{kl})$ are real valued and
$(g^{kl})$ and $(e^{kl})$ are symmetric.

Then simple scale invariant assumptions for the coefficients would be
\begin{equation} \label{g0} \Vert d \Vert_{L^\infty} + \| (t+x^2)^\frac12 \d_x
  d\|_{L^\infty} + \| t \d_t d\|_{L^\infty}\ll 1 
\end{equation}  
Here and in the sequel $d$ stands for a generic coefficient of the
form $ g^{kl}-\delta^{kl}$, $d^{kl}$ and $ e^{kl}$. For $V$ and $W_{1,2}$ we
could consider conditions of the form
\begin{equation}\label{v0}
  \Vert V \Vert_{L^1L^{\infty} + L^\infty L^{n/2}} \ll 1,
\end{equation}
\begin{equation}\label{W0}
  \Vert  W_{1,2} \Vert_{L^2L^\infty + L^\infty L^n} \ll 1 .
\end{equation} 
Here and in the sequel we use the notation $L^p L^q= L^p_t L^q_x$.

The situation is however more complex and we may take \eqref{g0} to
\eqref{W0} only as guidelines. We will have to strengthen \eqref{g0}
to include some dyadic summability.  on the other hand we are able to
weaken the time differentiability to a $C^{\frac13}$ Holder condition
on small time scales. 

We are also able to slightly weaken \eqref{v0} almost to uniform
bounds on dyadic sets. However, we are unable to use mixed norms for
$W_{1}$ and $W_2$, and we restrict ourselves to a summable $L^{n+2}$
norm in dyadic sets.

To state our assumptions on $g,V$, $W_{1}$ and $W_2$ we consider a double
infinite dyadic partition of the space,
\begin{equation}\label{aij1} 
\R^+ \times \R^n = \bigcup_{i=-\infty}^\infty \bigcup_{j=0}^\infty
A_{ij}
\end{equation} 
where
\begin{equation}\label{aij2}  
A_{ij} = \{(t,x) \in \R^+ \times \R^n\ |\ e^{-4i-4} \leq t \leq
e^{-4i}, \ e^j \leq 1+ 2|x| t^{-\frac12} \leq e^{j+1}\}.
\end{equation} 
Consider the subset of indices
\begin{equation}\label{matA}
  \mathcal{A} = \{ (i,j): j \le 2i +2\} 
\end{equation} 
defining a partition of the cylinder 
\[
Q = [0,1] \times B(0,1)
\]
  Define
\begin{equation} \label{matAtau} \mathcal{A}(\tau) = \{ (i,j)\in
  \mathcal{A} : 4i \ge \ln \tau+1, j \le \frac12 \ln \tau +2\}
\end{equation} 
which corresponds to a partition of the cut parabola
\[
Q_\tau  = \{ (t,x): |x|^2 \le \tau t \le 1 \}.
\]

\begin{figure}[here]
  \begin{tikzpicture}[yscale=.15]
    \draw[->] (-5,0)--(5,0) node[anchor=north west]{$x$}; \draw[->]
    (0,0)--(0,18) node[anchor=south east]{$t$};
    \fill[blue!20](-3.5,12.25) parabola bend (0,0) (3.5,12.25) --
    cycle; \draw(-4,16) parabola bend (0,0) (4,16); \draw(-3.5,12.25)
    -- (3.5,12.25); \draw[color=gray,very thin] (3.5,12.25) --
    (4,12.25) node[anchor=west,color=black]{$t =
      \tau^{-1}$};\draw(3.5,.9)--(3.5,-.9)
    node[anchor=north]{$|x|=1$};
  \end{tikzpicture}
\caption{The cut parabola}
\end{figure}
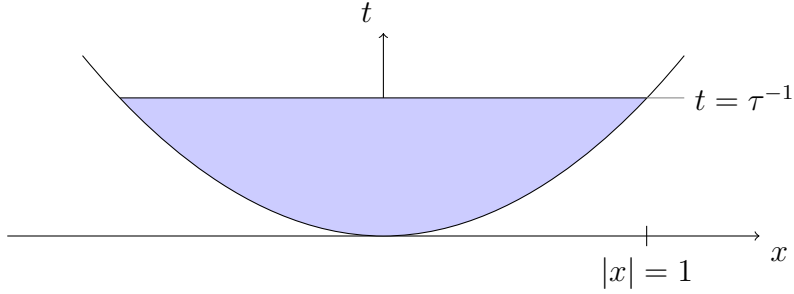

We also consider a decomposition of $Q$ into dyadic time slices
\[
A_i =  [e^{-4i-4}, e^{-4i}] \times B(0,1)
\]
and a similar partition of the cut parabola $Q_\tau$ into the sets
\[
A_i^\tau =  A_i  \cap Q_\tau
\]

Given a function space $X$ and $1 \le q < \infty$ we introduce the
Banach spaces $l^q(\mathcal{A}, X)$ with norms
\[
\|V\|_{l^q(\mathcal{A}, X)}^q = \sum_{i,j\in \mathcal{A}}
\|V\|_{X(A_{ij})}^q.
\]
In a similar manner we define the spaces $l^{\infty}(\mathcal{A},X)$.

Within the sets $A_{ij}$ we define the modulus of continuity  $(m_{ij})$
in time
\[
m_{ij} (\rho) = e^{4i} \rho + e^{\frac{2}3(2i-j)}  \rho^\frac13
\]
and denote by  $C^{m_{ij}}_t$  the space of continuous functions with 
finite seminorm
\[
\| u\|_{C^{m_{ij}}_t} = \sup_{t_1,t_2,x} \frac{|u(t_1,x) -u(t_2,x)|}{m^{ij}(|t_1-t_2|)}
\]
For the reader's convenience we note that within $A_{ij}$ we have
$ e^{4i} \approx t^{-1}$ and $e^{\frac{2}3(2i-j)} \approx (t+|x|^2)^{-\frac13}$. 

For the coefficients of the operator $P$ in \eqref{pgen} we
change the condition \eqref{g0} to 
\begin{equation}
  \label{g}
  \sup_\tau 
  \sum_{\mathcal{A}(\tau)}    \Vert d \Vert_{L^\infty(A_{ij})} +
  e^{j-2i}   \Vert d \Vert_{\lip_x(A_{ij})}
  + \|d\|_{C^{m_{ij}}_t(A_{ij})}
  \ll 1. 
 \end{equation}
where we note that $ e^{j-2i} \approx (t+x^2)^\frac12$ in $A_{ij}$.

The pointwise bound for $g-\id$ in \eqref{g}, namely
\begin{equation}
  \sup_\tau   \Vert g-\id \Vert_{l^1(\mathcal{A}(\tau), L^\infty)} \ll 1 
\label{g1}\end{equation}
is not really needed for our results. It can be always obtained from
the other bounds after  a  change of coordinates. This is
discussed in the appendix.

The assertion \eqref{g} is satisfied for $g\in \lip_x \cap
C^\frac13_t$ provided that $g(0,0)= \id$. Indeed by scaling we may
assume that the $ \lip_x \cap C^\frac13_t$ norm is small therefore it
suffices to compute
\[
\Vert (t+|x|^2)^\frac13 \Vert_{l^1(\mathcal{A} (\tau),L^\infty)} \le \sum_{i \ge
  \ln \tau} \sum_{j \le \ln(\tau)/2+2} e^{\frac13(j-2i)} \lesssim 1.
\]

For the potential $V$ we consider:
\begin{equation}
  \label{V}
  \begin{split} 
    \Vert V \Vert_{l^\infty(\mathcal{A}, L^1 L^\infty+ L^\infty L^{n/2})} & \ll
    1\quad \text{ for } n > 2
    \\
    \Vert V \Vert_{l^\infty(\mathcal{A}, L^1 L^\infty+ L^p L^{p'})} & \ll 1\quad
    \text{ for } n > 2, 1 \le p < \infty,
    \\
    \Vert V \Vert_{l^\infty(\mathcal{A},L^1 L^\infty+ L^2 L^{1})} & \ll 1\quad
    \text{ for } n =1
  \end{split}
\end{equation}
where $p'$ in the second line is the dual exponent. In addition 
we require that
\begin{equation}\label{V2}
  \sup_{i\geq 0} 
\Vert \chi_i   V \Vert_{L^1L^\infty+L^\infty L^{n/2} } \ll 1 \qquad n > 2  
\end{equation} 
with the obvious modifications for $n=1,2$, where 
$\chi_i$ is the characteristic function of the
set
\[
\{(t,x): e^{-4i-4} \le t \le e^{-4i},\  t^{-1/2} |x| \le i \}.
\]
Both \eqref{V} and \eqref{V2} are fulfilled if $V \in
L^1L^\infty+L^\infty L^{n/2}$ with small norm.

Finally for the gradient potentials $W_{1,2}$ we introduce the summability
condition with respect to time slices
\begin{equation}
  \label{W} 
 \sup_\tau \sum_i \Vert W_{1,2} \Vert_{L^{n+2}(A_i^\tau)} \ll 1.
\end{equation}
As a consequence of this we note the uniform bound
\begin{equation}
  \label{W2}
   \sup_i  \Vert W_{1,2} \Vert_{L^{n+2}(A_i)} \ll 1.
\end{equation}

Now we can state our main results.

\begin{theorem}
  Let $P$ be as in \eqref{pgen} with coefficients satisfying \eqref{g}.
  Assume that the potentials $V$ and $W_{1,2} $ satisfy \eqref{V}, \eqref{V2}
  and \eqref{W}.  Then \SUCP holds at $(0,0)$ for $H^1$ functions $u$
  satisfying \eqref{eq}.
  \label{tmain}\end{theorem}

It is part of the conclusion that the trace of $u$ at $t=0$ exists
near $x=0$.  The assumptions on the operator seem to be too weak to
imply existence of a trace in general. More precisely shall prove
\[
\Vert u(t,. ) \Vert_{L^2(B(0,1/8))} \lesssim e^{-\frac\delta{t}}
\]
for some $\delta >0$.  

The $C^{\frac13}$ H\"older regularity in time for the metric $g$ seems
so be new, improving the $C^\frac12$ H\"older regularity in
\cite{MR1971939}. It is not clear to the authors whether this
condition is optimal or not.

$ \sim \approx $

We may replace the assumptions by stronger translation invariant
assumptions,  
\begin{equation}
\| g \|_{Lip_x} + \|g\|_{C^\frac13_t} \lesssim 1
\label{gb}\end{equation}
\begin{equation}
\| V\|_{L^1L^\infty + L^\infty L^{n/2}} \ll 1
\label{Vb}\end{equation}
\begin{equation}
\sum_i \|W_{1,2} \|_{L^{n+2}(A_i)} \lesssim 1
\label{Wb}\end{equation}
  Then we also obtain a stronger conclusion.
\begin{theorem}\label{tmain2}
  Let $P$ be as in \eqref{p} with coefficients as in \eqref{gb}.
  Assume that the potentials $V$ and $W$ satisfy \eqref{Vb}
  respectively \eqref{Wb}.
 Then  \SUCPv holds at $(0,0)$ for $H^1$ functions $u$
  satisfying \eqref{eq}.
\end{theorem}

We remark that the condition \eqref{gb} is really too strong, and that
with some additional work (see Remark~\ref{goodg}) one can bring it 
almost to the level of \eqref{g}. Precisely, it suffices to replace
\eqref{g} by
\begin{equation}
  \label{gtwo}
  \sup_\tau 
  \sum_{\mathcal{A}(\tau)}    \Vert d \Vert_{L^\infty(A_{ij})} +
  e^{j-2i}   \Vert d \Vert_{\lip_x(A_{ij})}
  + \|d\|_{C^{m_{ij}^2}_t(A_{ij})}
  \ll 1. 
 \end{equation}
where the slightly stronger time continuity modulus  $m_{ij}^2$ is
given by
\[
m_{ij}^2 (\rho) =  e^{4i-2j}  \rho + 
e^{\frac{2}3(2i-j)}  \rho^\frac13
\]

However, we cannot  keep the additional terms in \eqref{pgen},
because we need to be able to meaningfully talk about the trace of the
solution at time $t=0$. 

Both theorems are consequences of quantitative estimates, which also imply {\sl
  weak unique continuation} under the assumptions of Theorem \ref{tmain2}: 
\[
\parbox[c]{4in}{\it Let $u$ be a solution to (\ref{eq}) for which 
$u(t_0,.)$ vanishes in the closure of an open set.
Then $(u(t_0,.)$  vanishes in a neighborhood of the closure. } \hspace{.5in} \UCP
\]

If
$u$ satisfies the assumptions and vanishes in an open set $U$, then it
vanishes in the time slices $t=t_0$ in an open neighborhood of the closure of
$U_{t_0} = \{ (x: (t,x) \in U\}$.

Theorems \ref{tmain} and \ref{tmain2} are  nontrivial consequences of a
Carleman inequality.  To state a first version of the Carleman inequality we
introduce an additional family $\mathcal{B}(\tau)$ of sets which is a
partition of the cylinder $[0,\tau^{-1}) \times B(0,1)$, consisting of
\begin{equation} \label{parB1} A_{ij}, \qquad \ln \tau \le 4i \le
  \tau^{1/2}, 0 \le j \le \ln \tau/2+2,
\end{equation} 
\begin{equation} \label{log} [e^{-4i-4}, e^{-4i}] \times B(0, e^{-2i}
  \tau^{1/2}), \qquad 4i >\tau^\frac12,
\end{equation} 
\begin{equation} \label{parB2} A_{ij}, \qquad \ln \tau \le 4i,\ \ \ln
  \tau/2 \le j \le 2i.
\end{equation}   

This partition is coarser than the partition of the same cylinder into
the sets $A_{ij}$. This is the reason why we need the assumption
\eqref{V2}.  More precisely Assumptions \eqref{V} and \eqref{V2} imply
\begin{equation}\label{V3}
  \Vert V \Vert_{l^\infty ( \mathcal{B}(\tau), 
    L^1 L^\infty + L^\infty L^{n/2})} \ll 1.
\end{equation}

\begin{theorem} \label{carleman} Let $\tau_0 \gg 1$, $\varepsilon
  >0$ and $P$ as in \eqref{pgen} with coefficients satisfying
  \eqref{g}. Suppose that $W_{1,2}$ satisfY \eqref{W} 
   with constants depending on $\varepsilon$ and $\tau_0$.  Then
  there exists $C>0$ such that for all $\tau \ge \tau_0$ the following
  is true: Suppose that $v \in L^2(H^1)$ is compactly supported in $
  [0,8\tau^{-1})\times B(0,2)$ and that it vanishes of infinite order near $(0,0)$.
  Then we can find a function $\phi \in C^\infty([0,8\tau^{-1}] \times
  B(0,1) \backslash \{ 0,0\})$ and $h \in C^\infty(\R^+)$ which satisfy
  \begin{equation}\label{hprime}
    \tau \le h^\prime \le (1+\e) \tau 
  \end{equation}
  \begin{equation}\label{weight}
\left | \phi(x,t) - \left(h(-\ln t) -\frac{x^2}{8t}\right)\right|
\leq \e \left( \tau+ \frac{x^2}t\right)
  \end{equation}
  such that
  \[ \Vert e^\phi v \Vert_{l^2(\mathcal{B}(\tau),L^\infty L^2\cap L^2
    L^{\frac{2n}{n-2}}) )} \le C \Vert e^\phi ( P  + W_1 \nabla + \nabla W_2) v
  \Vert_{l^2(\mathcal{B}(\tau), L^{1} L^{2}+ L^2
    L^{\frac{2n}{n+2}})},
  \]
for $n \geq 3$, respectively
  \[ \Vert e^\phi v \Vert_{l^2(\mathcal{B}(\tau),L^\infty L^2\cap
    L^{p'} L^{q'} )} \le C \Vert e^\phi ( P  + W_1 \nabla +\nabla W_2) v
  \Vert_{l^2(\mathcal{B}(\tau), L^{1} L^{2}+ L^{p} L^{q})},
  \]
  for $n=2$, $\frac1p + \frac1q = \frac12$, $2<p$, and
  \[ \Vert e^\phi v \Vert_{l^2(\mathcal{B}(\tau),L^\infty L^2\cap L^4
    L^\infty )} \le C \Vert e^\phi ( P  + W_1 \nabla + \nabla W_2 )v
  \Vert_{l^2(\mathcal{B}(\tau), L^{1} L^{\infty}+ L^{4/3} L^1)}
  \]
for $n=1$.
\end{theorem}

The statement of the Carleman inequality is involved for several
reasons. The weight $t^{-\tau} e^{-|x|^2/8t}$, which works for the
constant coefficient case, has to be modified so that it has more
convexity in order to handle variable coefficients, spatial
localization and the gradient potential. However the polynomial growth
in time (imposed by the assumption of vanishing of infinite order)
limits the available amount of convexity; this is the origin of the $l^1$
summability in \eqref{g1}, \eqref{W2}, and to a lesser extend of
\eqref{V2}.
 
If $W=0$ then the Carleman inequality holds for a large explicit class
of weights $e^\phi$. This cannot be possibly true for general gradient
potentials.  Instead, we are only able to prove that there exists some
weight function $\phi$, which now depends on $\tau$, $u$ and $W$, for
which the uniform Carleman inequality holds. This strategy goes back
to the seminal work of T. Wolff \cite{780.35015} and has been used by
the authors for the elliptic problem \cite{MR2001m:35075}.

The partition $A_{ij}$ is much finer than the dyadic decomposition in
$t$ only, which would correspond to the dyadic decomposition in the
elliptic case. We are only able to localize the estimates to the sets
$A_{ij}$ if we make the weight function sufficiently convex. We can do
this for many $A_{ij}$, but not for all of them.  The sets \eqref{log}
correspond directly to the assumption \eqref{V2}. We need to have
control of the $L^1 L^\infty+ L^\infty L^{n/2}$ norm of $V$ in sets
which are not smaller than those of the partition in \eqref{parB1},
\eqref{log} and \eqref{parB2}.

We have stated Theorem \ref{carleman} in a simpler form which suffices
to derive Theorem \ref{tmain} and Theorem \ref{tmain2}.  However the
full estimate we prove is stronger in that it also contains precise
$L^2$ bounds. These are essential for the localization and
perturbations techniques we use.

The strategy of the proof is the same as in \cite{MR2001m:35075}:

\begin{enumerate}
\item We construct families of pseudoconvex weights and derive $L^2$
  Carleman inequalities. The convexity of weights determines the
  space-time localization scales and the admissible size of
  perturbations.
\item We enhance the above $L^2$ Carleman inequalities to include
  $L^p$ estimates. Due to the $L^2$ localization it suffices to do
  this in small sets.  This allows us to use perturbation arguments
  starting from the case of the heat equation with the weight
  $t^{-\tau}e^{-|x|^2/8t}$.
\item $L^p$ estimates for the spectral projections to spherical
  harmonics imply the $L^p$ Carleman inequalities in the elliptic
  case. Here spectral projection for the Hermite operator play a
  similar role.
\item Finally we include Wolff's osculating argument into the scheme
  in order to handle the gradient potentials. The efficiency of this
  part depends on the flexibility in the choice of the weight
  functions.
\end{enumerate}

The complexity of the weights and the $L^2$ Carleman estimates comes
mainly from the geometry of the classical harmonic oscillator.  Orbits
are contained in a sphere in $\R^{2n}$. The projection down in the $x$
space is a ball, where frequency variables have a different behavior
in radial and angular directions and near the boundary of the the
ball. It turns out that our analytic estimates reflect these features.

\section{Proof of Theorem \ref{tmain} and \ref{tmain2} }

In this section we prove Theorem \ref{tmain} and \ref{tmain2} assuming
Theorem \ref{carleman}.  The relation between Carleman estimates and
unique continuation is fairly straightforward in the elliptic case. In
the parabolic situation the argument is less direct due to the more
complex geometry of the level sets of the weight functions.

It is a standard consequence of a localized energy inequality that for
the parabolic equation \eqref{eq} $u(t) $ and its gradient can be
controlled by $L^2$ norms of $u$.

\begin{proposition}\label{energy}
  Let $n \ge 3$ and suppose that $v$ solve the parabolic equation
  \[ v_t + \d_k a^{kl} \d_l v = W_1 \nabla v+ \nabla (W_2 v)  + V v \] on the space-time
  cylinder $Q=[0,2] \times B(0,2)$ with $a \in \lip$ uniformly
  elliptic and
  \[
  \Vert W_{1,2} \Vert_{L^{n+2}(Q)} + \Vert V \Vert_{L^1 L^\infty + L^\infty
    L^{n/2} } \ll 1
  \]
  Then
  \[ \sup_{0\le t \le 1} \Vert v(t) \Vert_{L^2(B(0,1))} + \Vert
  \nabla_x v \Vert_{L^2([0,1] \times B(0,1))}\lesssim \Vert v
  \Vert_{L^2(Q)}.
  \]
  If $n=2$ then the same statement is true with $L^\infty L^{n/2}$
  replaced by $L^pL^q$ with $1\le p < \infty$, $1<q \le \infty$ and
  $\frac1p + \frac1q = 1$.  Similarly, if $n=1$ we have to replace it
  by $L^4 L^\infty$.
\end{proposition}

Given our assumptions \eqref{g}, \eqref{V}, \eqref{V2}, \eqref{W} and
\eqref{W2} we can apply Proposition \ref{energy} rescaled in sets of
the form $[t_0,2t_0] \times B(x_0,t_0^{1/2})$, which are subsets of
the $A_{ij}$.  Summing up with respect to such sets contained in
a parabolic cube
\[
Q_r = [0,r^2] \times B(0,r)
\]
we obtain the following consequence.

\begin{cor} \label{vanish} The following estimate holds under the
  assumptions of Theorem \ref{tmain}
  \[
\| t^\frac12 u\|_{L^\infty L^2 (Q_r) } +  \| t^\frac12 \nabla u\|_{L^2 (Q_r) }
\lesssim  \| u\|_{L^2(Q_{2r})}
  \]
\end{cor}

\begin{proof}[Proof of Theorem \ref{tmain}] 
  We choose $\tau \ge \tau_0$ and $0 < \delta \ll
  \tau^{-1/2}$ and introduce the sets
 \[
  \begin{split}
    E_\delta = & \big( [0,2\delta^2] \times B(0,2\delta)\big)
    \backslash \big([0,\delta^2]
    \times B(0,\delta)\big) \\
    F^{ext}_\tau = & \big([0, 2/\tau ] \times B(0,2)\big)
    \backslash \big([0,1/\tau] \times B(0,1)\big) \\
    F^{int}_\tau = & [ 1/32\tau , 1/16\tau] \times B(0,1/4)
  \end{split}
  \]

  \begin{figure}
    \begin{tikzpicture}[scale=1.5]
      \fill[green!25] (-2.3,0)--(-2.3,2.3) -- (2.3,2.3)-- (2.3,0) --
      (1.8,0)--(1.8,1.9)-- (0,1.9) node[anchor=south
      east,color=black]{$F_\tau^{ext}$}--(-1.8,1.9)-- (-1.8,0) --
      cycle; \fill[red!30]
      (-.4,0)--(-.4,.2)--(.4,.2)--(.4,0)--(.15,0)--(.15,.08)--(-.15,.08)--(-.15,0)
      -- (-0.275,0) node[anchor=north,color=black]{$E_\delta$}
      --cycle; \fill[blue!35]
      (-1.0,0.5)--(-1.0,1.2)--(1.0,1.2)--(1.0,0.5) --
      (0,0.5)node[color=black,anchor=south east]{$F^{int}_\tau$} --
      cycle; \draw[->] (-2.7,0)--(2.7,0) node[anchor=north west]{$x$};
      \draw[color=gray!30,very thin] (2.3,2.3)--(2.6,2.3)
      node[anchor=west,color=black]{$2/\tau$};
      \draw[color=gray!30,very thin] (2.3,1.9)--(2.6,1.9)
      node[anchor=west,color=black]{$1/\tau$}; \draw[color=gray!30,
      very thin] (0,1.5)--(2.6,1.5)
      node[anchor=west,color=black]{$t_0$}; \draw[color=gray!30, very
      thin] (1.0,0.5)--(2.6,0.5)
      node[anchor=west,color=black]{$1/8\tau$}; \draw[color=gray!30,
      very thin] (1.0,1.2)--(2.6,1.2)
      node[anchor=west,color=black]{$1/16\tau$}; \draw
      (2.3,0.05)--(2.3,-0.05) node[anchor=north,color=black]{$2$};
      \draw (1.8,0.05)--(1.8,-0.05)
      node[anchor=north,color=black]{$1$}; \draw
      (1.0,.05)--(1.0,-0.05) node[anchor=north,color=black]{$1/4$};
      \draw (0.4,.05)--(0.4,-0.05)
      node[anchor=north,color=black]{$2\delta$}; \draw[->]
      (0,0)--(0,2.5) node[anchor=south east]{$t$};
    \end{tikzpicture}
\caption{The sets $E_\delta$, $F_\tau^{ext}$ and $F_\tau^{int}$}
  \end{figure}
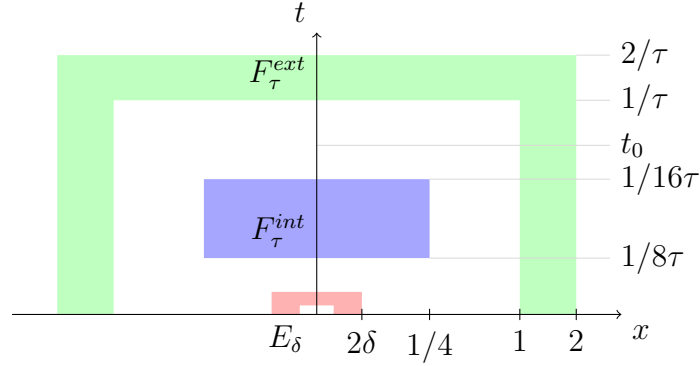

  Our strategy will be to truncate $u$ in $E_\delta$ and
  $F^{ext}_\tau$ and to apply Theorem \ref{carleman} to the truncated
  function in order to obtain a good bound on $u$ in $F_\tau^{int}$.

  Let $\eta$ be a cutoff function supported in $ [0,2)\times B(0,2)$
  and identically $1$ in $[0,1] \times B(0,1) $. For $\delta \ll
  \tau^{-1/2}$ we define
  \[
  v_\delta(t,x) = (1-\eta(t/\delta^2,x/\delta)) \eta(\tau t/8, x)u(x,t)
  \]
  which satisfies
  \[
  (P+W\nabla) v_\delta = V v_\delta - [P+W\nabla,
  \eta(t/\delta^2,x/\delta)] u + [P+W\nabla, \eta(\tau t/8,x)] u.
  \]
  The second term on the right hand side is supported in $E_\delta$
  and the third one in $F_\tau^{ext}$. 

  We apply Theorem \ref{carleman} to $v_\delta$. One should keep in
  mind that the corresponding weight $\phi$ depends on $\delta$ but
  that the bounds we prove are uniform with respect to $\delta$.  We
  normalize the function $h$ by $h(0)=0$.  We have to control the size
  of $\phi$ in the sets $E_\delta$, $F^{ext}_\tau$ and $
  F^{int}_\tau$.  Due to \eqref{hprime} we have
\[
\tau s \leq h(s) \leq 2 \tau s, \qquad s \geq 0
\]
By \eqref{weight} we obtain a rough polynomial bound in $\delta$
  \begin{equation} \label{Edelta} 
e^\phi \le t^{-2\tau}
    e^{-\frac{|x|^2}{8 t} + \e (\tau+\frac{x^2}t) } \le t^{1/2} c(\tau)
    \delta^{-4\tau-1} \qquad \text{ in } E_\delta.
  \end{equation} 
  Let
  \[
  M = \sup \{ e^{h(-\ln(t))} e^{-\frac{|x|^2}{8t} + \e (\tau+\frac{x^2}t)}:(t,x) \in
  F_\tau^{ext}\}
  \]
  By \eqref{hprime} the supremum is attained at a point $(t_0,x_0)$
  with $ \frac12 \le 8\tau t_0 \le 1$ and $|x_0| = 1$.  A simple
  computation also shows that
  \[
  \sup \{ t^{-1/2} e^{h(-\ln(t))} e^{-\frac{|x|^2}{8t} + \e(\tau+\frac{x^2}t)
    }:(t,x) \in F_\tau^{ext}\} \lesssim \tau^{1/2} M.
  \]
  Then $M$ dominates $ e^{\phi}$ in $F^{ext}_\tau$:
  \begin{equation} \label{Fext} e^{\phi} \le M, \qquad t^{-1/2}
    e^{\phi} \lesssim \tau^{1/2} M \qquad \text{ in } F_\tau^{ext}.
  \end{equation} 

  Next we need to bound $e^\phi$ from below in $F^{int}_\tau$ in terms
  of $M$,
  \begin{equation} \label{compare} \inf_{F^{int}_\tau} e^\phi \ge
    e^{\frac12\tau} M
  \end{equation} 
  To see this we compute for $(t,x) \in F^{int}_\tau$ and sufficiently
  small $\varepsilon$:
  \begin{eqnarray*}
  {\phi(t,x)-\phi(t_0,x_0)} & \ge&  h(-\ln t) - h(-\ln t_0) +
    \frac1{8t_0} - \frac1{32 t_0} - 2\e( \tau +\frac{1}{t_0})  
\\ &\ge& {( \frac3{4} -20 \e)
    \tau } \ge {\frac12{\tau}}
  \end{eqnarray*}
  and use \eqref{Edelta}
  and \eqref{Fext}.

 Theorem \ref{carleman} applied to $v_\delta$ yields
  \begin{equation} \label{carlem}
    \begin{split} 
      \Vert e^\phi v_\delta \Vert_{l^2(\mathcal{B}(\tau),
        L^2L^\frac{2n}{n-2}\cap L^\infty L^2)} \lesssim & \Vert e^\phi
      Vv_\delta \Vert_{l^2(\mathcal{B}(\tau),
        L^2L^{\frac{2n}{n+2}}+L^1L^2)} \\ & + \Vert e^\phi [P+W\nabla,
      \eta(t/\delta^2,x/\delta)] u \Vert_{l^2(\mathcal{B}(\tau),
        L^2L^{\frac{2n}{n+2}}+L^1L^2)} \\ & + \Vert e^\phi
      [P+W\nabla, \eta(\tau t,x)] u \Vert_{l^2(\mathcal{B}(\tau),
        L^2L^{\frac{2n}{n+2}}+L^1L^2)}
    \end{split}
  \end{equation}

  By H\"older's inequality we have
  \[
  \! \Vert e^\phi Vv_\delta \Vert_{l^2(\mathcal{B}(\tau),
    L^2L^{\frac{2n}{n+2}} +L^1L^2)} \! \lesssim \!  \Vert V
  \Vert_{l^\infty(\mathcal{B}(\tau), L^1 L^\infty + L^\infty L^{n/2})}
  \Vert e^\phi v_\delta \Vert_{l^2(\mathcal{B}(\tau),
    L^2L^{\frac{2n}{n-2}}\cap L^\infty L^2)}.\!
  \]
  Due to the smallness in \eqref{V3} we can absorb this term on the
  left hand side of the inequality.

  We calculate the first commutator
  \[
\begin{split}
f_\delta& =  [P+W\nabla, \eta(t/\delta^2,x/\delta)] u \\ &=  (
\frac{\d}{\d t} + \d_k g^{kl} \d_l + 2 \frac{x_k}{t} d^{kl} \d_l +W
  \nabla)\eta(t/\delta^2,x/\delta) 
\\ & + 2 \delta^{-1} g^{kl} (\d_k
  \eta)(t/\delta^2,x/\delta) \d_l u
\end{split}  
\]
By \eqref{Edelta} we have
\[
\| e^\phi f_\delta \|_{ l^2(\mathcal{B}(\tau),
        L^2L^{\frac{2n}{n+2}}+L^1L^2)}\lesssim c(\tau) \delta^{-4\tau-1} 
\| t^\frac12   f_\delta \|_{ l^2(\mathcal{B}(\tau),
        L^2L^{\frac{2n}{n+2}}+L^1L^2)}
\]
For the $W$ term we use \eqref{W} and Holder's inequality. For term
involving $d^{kl}$ we bound the $L^1 L^2$ norm in terms of an
$L^\infty L^2$ norm, using \eqref{g} which implies that the pointwise
bound for $d^{kl}$ is summable with respect to dyadic time regions.
For the remaining terms we simply bound the $L^1L^2$ norm in terms of
the $L^2$ norm.  This yields
\[
\begin{split}
\| e^\phi f_\delta \|_{ l^2(\mathcal{B}(\tau),
        L^2L^{\frac{2n}{n+2}}+L^1L^2)} \lesssim&  c(\tau)
      \delta^{-4\tau-2} 
(   \|  u \|_{L^2(E_\delta)} +  
\|(\partial_x \eta) (t/\delta^2,x/\delta)  t^\frac12 \nabla u\|_{L^2} 
\\ &+  \| (\partial_x \eta) (t/\delta^2,x/\delta)  t^{\frac12} 
u\|_{L^\infty L^2})
\end{split}
\]
Then we can apply a straightforward modification of
Corollary~\ref{vanish} on the $E_\delta$ scale to finally obtain
\[
  \Vert e^\phi [P+W\nabla, \eta(t/\delta^2,x/\delta)] u
  \Vert_{l^2(\mathcal{B}(\tau),
        L^2L^{\frac{2n}{n+2}}+L^1L^2)}
  \lesssim c(\tau) \delta^{-4\tau-2} \Vert u \Vert_{L^2(E_\delta)}.
  \]
    Similarly we can estimate the second commutator
  \[
  \Vert e^\phi [P+W\nabla, \eta(\tau t,x)] u
  \Vert_{l^2(\mathcal{B}(\tau), L^2L^\frac{2n}{n+2}+ L^1 L^2)}
  \lesssim M \tau^{1/2} \Vert u \Vert_{L^2(F^{ext}_\tau)}.
  \]
  Hence by inequality \eqref{carlem} we get 
  \begin{equation}
 \Vert e^\phi v_\delta
    \Vert_{l^2(\mathcal{B}(\tau), L^2L^\frac{2n}{n-2}\cap L^\infty
      L^2)} \lesssim M \tau^{1/2}\Vert u \Vert_{L^2(F^{ext}_{\tau})} +
  c(\tau) \delta^{-4\tau-2} \Vert u \Vert_{L^2(E_\delta)}  .
  \end{equation} 
  Within $F_\tau^{int}$ we have $v_\delta = u$. Then by \eqref{compare} 
we obtain
  \begin{equation} \label{weighted} 
\Vert u \Vert_{L^\infty
      L^2(F^{int}_\tau)} \lesssim \tau^{1/2} e^{-\frac12 \tau} \Vert u
    \Vert_{L^2(F^{ext}_\tau)} + c(\tau) \delta^{-4\tau-2} \Vert u
    \Vert_{L^2(E_\delta)}.
  \end{equation}  
  Also by the vanishing of infinite order the second term tends to
  zero as $\delta \to 0$.  Hence as $\delta \to 0$ we obtain
\begin{equation} \label{weighted1} \Vert u \Vert_{L^\infty
      L^2(F^{int}_\tau)} \lesssim \tau^{1/2} e^{-\frac12 \tau} \Vert u
    \Vert_{L^2(F^{ext}_\tau)} 
  \end{equation}  
  For $0<t\ll 1$ we choose $\tau = \frac1{16 t}$ to obtain
  \[ \Vert u(t,.) \Vert_{L^2(B(0,1/4))} \lesssim t^{-1/2} e^{-
    \frac1{32 t}}. \] This completes the proof of Theorem \ref{tmain}.
\end{proof}

\begin{proof}[Proof of Theorem \ref{tmain2}] 
  We extend the potentials $V$ and $W$ by zero to negative time, and
  $g^{kl}(t,x)= g^{kl}(0,x)$ for $t < 0$. By definition, possibly
  after rescaling, we have $ u(0,.) \in L^2(B(0,2))$. We now solve the
  mixed problem
  \begin{equation}
    \label{extend}  u_t + \d_k g^{kl} \d_l u = 0 \qquad \text{ for } t < 0 
    \text{ and } |x| \le 2 
  \end{equation} 
  with the boundary condition
  \[ u(t,x) = 0 \qquad \text{ if } |x| = 2 \text{ and } t < 0 \] and
  the obvious initial condition to obtain an extension of $u$ to
  negative. The heat kernel for \eqref{extend} satisfies Gaussian
  estimates.  In particular we obtain from \eqref{vanishing} for all
  positive integers $N$ with a constant $c_N$ possibly differing from
  \eqref{vanishing}
  \begin{equation}
    \int_{-r^2}^0 \int_{B_r(0)} |u|^2 dx dt \lesssim c_N^2 r^{2N}. 
 \label{vanish+} \end{equation} 

We seek to prove that the bound \eqref{weighted1} still holds in this
context. The difficulty is that we only know that $u$ vanishes of
infinite order at $(0,0)$ for negative time. To account for this we
shift the time up, $t \to t+2\delta$.  Arguing as in the previous
proof we obtain \eqref{weighted} with $u$ replaced by
$u(t+2\delta,x)$. Letting $\delta \to 0$ by \eqref{vanish+} we 
obtain \eqref{weighted1} and conclude as above.  

\begin{remark}
If one wants to prove Theorem \ref{tmain2} under the weaker 
assumptions on $g$ in \eqref{gtwo} then the origin needs to be avoided
in the above argument. Hence the time translation needs to be
accompanied by a spatial translation, namely
\[
u_\delta(t,x) = u( t+ 2\delta^2, x -  8 \tau \delta e_1)
\]
This translation places the image of the origin, or better of the
cube $[0,\tau \delta^2] \times B(0,  4 \tau \delta)$, within
the region $\{ \tau t < x^2 \}$. But in this region the conjugated
operator $P_\psi$, introduced later, is elliptic so only pointwise bounds
for $g$ are needed for the Carleman estimates.
\label{goodg}\end{remark}

\end{proof}

\section{ $L^2$ bounds in the flat case and the Hermite operator}
\label{sl2flat}

 In this section we prove the simplest possible $L^2$ Carleman
 estimate for the constant coefficient  backward parabolic equation
\[
 \d_t u + \Delta_x u= f
\]
This serves as a good pretext to introduce the class of weight
functions which is later modified for the variable coefficient case.

We also describe the change of coordinates which turns the 
backward parabolic operator  into a forward parabolic equation
for the Hermite operator $H$. In this way we are able to relate the 
$L^2$ Carleman estimates for the heat operator to spectral information 
for $H$.

\begin{proposition}
Let $u \in L^2$ with compact support away from $0$. Then
\begin{equation}
\| t^{-\tau-\frac12} e^{-\frac{x^2}{8t}} u\|_{L^2}
\leq \| t^{-\tau+\frac12} e^{-\frac{x^2}{8t}} (\d_t + \Delta)  u\|_{L^2}
\label{l2flat}\end{equation}
uniformly with respect to $\tau$ away from $(2n+\N)/4$.
\end{proposition}

\begin{proof}
In $\R^+ \times \R^n$ we introduce new coordinates $(s,y) \in \R\times
\R^n$ defined by
\begin{equation}
  \label{coord}
\left\{  \begin{array}[l]{l}
t=  e^{-4s} \cr
x= \frac{1}{2} e^{-2s} y
  \end{array} \right.
\end{equation}
Then 
\[
\frac{\d}{\d s} = -4e^{-4s} \frac{\d}{\d t} - e^{-2s} y 
\frac{\d}{\d x}
, \ \ \ 
\frac{\d}{\d y}= \frac{1}{2} e^{-2s} \frac{\d}{\d x}
\]
Hence in the new coordinates our operator becomes
\[
4 t (\d_t + \Delta_x) =  - \frac{\d}{\d s} - 2 y\frac{\d}{\d y}
+  \Delta_y
\]
If we conjugate it by $t^{n/4}e^{-\frac{x^2}{8t}} = e^{-ns} e^{-\frac{y^2}{2}}$
we obtain
\[
4 t^{1+n/4}  e^{-\frac{x^2}{8t}} (\d_t+\Delta_x)t^{-n/4}e^{ \frac{x^2}{8t}} =
 - \frac{\d}{\d s} 
+  \Delta_y - y^2 =: -\d_s -H =: -P_0 
\]
where $H$ is the Hermite operator
\[
H = -\Delta_y + y^2
\]
Then it is natural to define the new functions  
\[
v(s,y) = e^{-ns} e^{-\frac{y^2}2} u(e^{4s}, \frac{1}{2} e^{2s} y), \qquad 
g(s,y) = e^{(-n-4)s} e^{-\frac{y^2}2} f(e^{4s}, \frac{1}{2} e^{2s} y)
\]
which are related by
\[
P_0  v = g \Longleftrightarrow   (\d_t + \Delta )  u = f   
\]
In the new coordinates, the bound \eqref{l2flat} becomes 
\begin{equation}
\| e^{4 \tau   s} v\|_{L^2} \lesssim \| e^{4 \tau s} P_0 v \|_{L^2}. 
\label{piy}\end{equation}

Denoting $w = e^{4 \tau s} v$, we conjugate
\[
 e^{4 \tau s} P_0 v  =  e^{4 \tau s} P_0 e^{-4 \tau s} w = 
(- \d_s - H + 4\tau) w
\]
and the above bound becomes
\begin{equation}\label{ttau}
\| w \|_{L^2} \lesssim \| (- \d_s - H + 4\tau) w\|_{L^2}. 
\end{equation}
Since $\d_s$ and $H-4\tau$ commute we expand 
\[ 
\| (- \d_s - H + 4\tau) w\|_{L^2}^2 = \| \d_s w \|_{L^2}^2 +
 \| (H-4\tau)w \|_{L^2}^2  \ge d(4\tau, n+\mathbb{N})^2  \Vert w \Vert_{L^2}^2 .
\] 
 Note the spectral gap, which is essential in
order to obtain strong unique continuation results.
\end{proof}

For later use we also record the following slight generalization 
of the above result. For expediency this is stated in the $(y,s)$
coordinates, i.e. in the form of an analogue of \eqref{piy}. 

\begin{proposition}
Let $h$ be an increasing, convex,  twice  differentiable function 
so that 
\[
 d(h^\prime, \mathbb{N}) + h^{\prime\prime} \ge \frac14. 
\]
Then
\begin{equation}
\| (1+h'')^{1/2} e^{h(s)} v\|_{L^2} +  \left\|
\min\Big\{1, \frac{(1+h'')^{1/2}}{1+h^\prime} \Big\}
 e^{h(s)} H v\right\|_{L^2}
\lesssim  \|  e^{h(s)} (\d_s - H)  v \|_{L^2}
\label{l2flatp}\end{equation}
for all compactly supported $v \in L^2$.
\label{bell}\end{proposition}

\begin{proof}
After the substitution $w=e^{h(s)} v$ the bound \eqref{l2flatp}
becomes 
\begin{equation}
\|  (1+h'')^{1/2} w \|^2_{L^2} +  \left\|\min\Big\{1, \frac{(1+h^{\prime\prime})^{1/2}}{1+h^\prime}\Big\}       H w\right\|^2_{L^2}
\lesssim  \|  (\d_s -H + h'(s)) w \|^2_{L^2}.
\label{l2flatp1}\end{equation}
and we obtain the $L^2$ estimate through expanding the term on the right hand side with respect to its selfadjoint and skewadjoint part:  
\[ 
\begin{split} 
 \|  (\d_s -H + h'(s)) w \|^2_{L^2} =& \Vert \d_s w \Vert_{L^2}^2 
+ \Vert (-H + h^\prime) w \Vert^2_{L^2} + 
\Vert (h^{\prime\prime})^{1/2}  w  \Vert_{L^2}^2 
\\ \ge & \int   (d(h^\prime, \mathbb{N})^2 + h^{\prime\prime}) \|w(s)\|^2_{L^2}ds 
\end{split} 
\]
To complete the proof we observe that for each $s$ we have
\[ 
 \Vert H w(s)  \Vert_{L^2} \lesssim 
\Vert (-H + h^\prime)w(s)  \Vert_{L^2} +
 h^\prime(s) \Vert w(s) \Vert_{L^2}   
\]

\end{proof}

\section{Resolvent bounds  for the Hermite operator }
\label{resolvent} 
As seen in the previous section, the spectral properties of the 
Hermite operator play an essential role even in the simplest $L^2$ 
Carleman estimates for the heat equation. In this section 
we take a look at $L^2$ and $L^p$ bounds for its spectral projectors 
and its resolvent.

The spectrum of $H$ is $n+2\N$, and its eigenfunctions are the Hermite 
functions defined by
\[
u_\alpha = c_\alpha (\partial_y -y)^\alpha e^{-\frac{x^2}2}, \qquad
H u_\alpha = (n+2|\alpha|) u_\alpha
\]
As $|\alpha|$ increases, so does the multiplicity of the eigenvalues.
We denote the spectral projectors by $\Pi_\lambda$ for $\lambda \in
n+2\N$. We consider both the spectral projectors and the resolvent of
$H$ and obtain both $L^p$ and localized $L^2$ bounds.  

\subsection{Weighted $L^2$ bounds} 

  We consider two parameters
\[
1 \leq d, R \lesssim \lambda^{\frac12}
\]
We denote
\[
B_R = \{ y : |y| < R\}, \qquad, B^j_d=  \{ y : |y_j| < d\}, \ \ j= 1,
...,n
\]
By $\chi_R$, respectively $\chi^i_d$ we denote bump functions in
$B_R$, respectively $B^i_d$ which are smooth on the corresponding
scales.

\begin{proposition}
 The spectral projectors $\Pi_\lambda$  satisfy the localized $L^2$ bounds
\begin{equation}
  R^{-\frac12} \lambda^{\frac14}  \|\chi_R \Pi_\lambda f \|_{L^2} + 
R^{-\frac12}\l^{-\frac14} \|\chi_R  \nabla \Pi_\lambda f
  \|_{L^2} \lesssim \|
  f\|_{L^2}^2,
\label{gr}\end{equation}
respectively
\begin{equation}
d^{-\frac12}  \||D_j|^\frac12 ( \chi^j_{d} \Pi_\lambda f)\|_{L^2} 
\lesssim  \| f\|_{L^2}
\label{gd}\end{equation}
\label{pl2proj}\end{proposition}

\begin{proof}
  The inequality \eqref{gr} is trivial unless $ R \ll
  \lambda^\frac12$.  To prove it in dimension $n=1$ we only
  need to consider the case when $f$ is a Hermite function,
\[
f = \Pi_\lambda f = h_\lambda
\]
in which case it follows from the pointwise bound
\[ 
\lambda^{-\frac14} | h'_\lambda(x) | + \lambda^{\frac14} |h_\lambda(x)|
 \lesssim   \Vert h_\lambda \Vert_{L^2}, \qquad |x| \leq \frac12 \lambda^{\frac12}
\] 

In dimension $n=1$ \eqref{gd} follows by interpolation from \eqref{gr} with $R=d$.
This extends trivially to higher dimension by separation of variables.

It remains to prove \eqref{gr} in higher dimensions. Summing up 
\eqref{gd} with $d = R$ over $j$ we obtain  the bound
\[
R^{-\frac12} \||D|^\frac12 \chi_R \Pi_\lambda f \|_{L^2} \lesssim \|f\|_{L^2}^2
\]
For $|x| \lesssim R \ll \lambda^{\frac12}$ we have $|\xi|^2 \approx
\lambda$ in the characteristic set of $H -\Re z$, therefore the above
norm should essentially control the left hand side of \eqref{gr}.
For later use we prove a slightly more general result, which
in particular concludes the proof of \eqref{gr}.
\begin{equation}
 \lambda^{\frac14}  \| v \|_{L^2} \! + \l^{-\frac14} \|  \nabla v  \|_{L^2}\!
\lesssim  \||D|^\frac12  v \|_{L^2} 
+\||y|^\frac12 v\|_{L^2}+ \| (H-\lambda) v\|_{y L^2 + \nabla L^2 +
  \lambda^\frac12 L^2}\!
\label{ellp}\end{equation}
Indeed the norm on the right is equivalent to 
\[
\| H^{\frac14} v\|_{L^2} + \|(H+\lambda)^{-\frac12}  (H-\lambda) v\|_{L^2}
\gtrsim  \lambda^{\frac14}  \| v \|_{L^2} + \l^{-\frac14} \|H^\frac12 v\|_{L^2}
\]
In our case we apply \eqref{ellp} to $v = R^{-\frac12} \chi_R
\Pi_\lambda f$. Then
\[
\||y|^\frac12 v\|_{L^2} \lesssim \|\Pi_\lambda f\|_{L^2}
\]
while
\[
(H-\lambda) v = 2 R^{-\frac12} \nabla (\nabla \chi_R \Pi_\lambda f) +
R^{-\frac12} \Delta \chi_R \Pi_\lambda f
\]
which yields
\[
\|(H-\lambda) v \|_{y L^2 + \nabla L^2 +\lambda^\frac12 L^2} \lesssim R^{-\frac32} \|f\|_{L^2}
\]
 \end{proof}

To state the corresponding resolvent bounds we define the  spaces $\tilde X_2(z)$ by  
\[
\| u \|_{\tilde X_2(z)} = (1+|\Im z|)^\frac12 \| u \|_{L^2} +  \|
(H-z) u\|_{y L^2 + \nabla L^2+|z|^\frac12 L^2}
+ \sup_{j,d} d^{-\frac12}  \||D_j|^\frac12  \chi^j_{d} u\|_{L^2} 
\]
and the corresponding dual spaces $\tilde X_2^*(z)$.
These spaces are larger than the corresponding ``elliptic'' spaces,
\begin{equation}
\|v\|_{X_2(z)} \lesssim |z|^\frac12 \|v\|_{L^2} + \| y v\|_{L^2} +
\|\nabla v\|_{L^2}
\label{smthe}\end{equation}
On the other hand by  extending  the bound \eqref{ellp} 
to complex $\lambda$ we obtain a counterpart of \eqref{gr},
namely
\begin{equation}
  R^{-\frac12} \lambda^{\frac14}  \|\chi_R u\|_{L^2} + 
R^{-\frac12}\l^{-\frac14} \|\chi_R  \nabla u
  \|_{L^2} \lesssim    \| u\|_{ X_2(z)}
\label{grz}\end{equation}
Finally, the result of \eqref{gd}
can be written in the following dual forms
\begin{equation}
  \| \Pi_\lambda f\|_{ X_2(\lambda)} \lesssim \|f\|_{L^2},
  \qquad 
\| \Pi_\lambda f\|_{L^2} \lesssim \|f\|_{ X_2^*(\lambda)}
\label{x2lrdual}\end{equation}

The  localized $L^2$ resolvent bounds have the form

\begin{proposition}
 Let $n \geq 2$, $z \in \C$ with $dist(z, n+2\N) \gtrsim 1$, and $1 \leq d \leq
 R \ll \Re z$. Then 
\begin{equation}
\| u\|_{\tilde X_2(z)} \lesssim \|(H-z) f\|_{\tilde X_2^*(z)}
\label{l2res}\end{equation}
where the $d$ component of norms is omitted in dimension $n=1$.
\label{pl2res}\end{proposition}

\begin{proof}

  We first note that the bounds \eqref{x2lrdual} almost imply
  \eqref{l2res} up to a logarithmic divergence. They do imply easily a
  bound for higher powers of the resolvent for $z$ away from
  the spectrum of $H$,
\begin{equation}
\| (H-z)^{-1-k}f\|_{\tilde X_2(z)} \lesssim (1+|\Im z|)^{-k}
\|f\|_{\tilde X_2^*(z)}, \qquad k \geq 1.
\label{pl2resk}\end{equation}
as well as 
\begin{equation}
\| u \|_{L^2} \lesssim (1+|\Im z|)^{-\frac12} \|(H-z)u\|_{\tilde
  X_2^*(z)}.
\label{2rd}\end{equation}
Hence it remains to show that
\begin{equation}
  \| u \|_{\tilde X_2( z)} \lesssim 
(1+|\Im z|)^{\frac12}
  \|u\|_{L^2} + \|(H-z) u \|_{\tilde X_2^*(z)}.
\label{eeeqa}\end{equation}

Using a positive commutator technique we first prove
a one dimensional estimate. For this we define the 
one dimensional skewadjoint pseudodifferential operator\footnote{As defined
  the symbol of $Q$ is not smooth at $0$. However, any smooth 
modification in the ball $\{ x^2+\xi^2 < r^2\}$ will do.}
\[
Q_r = i Op^w(\tsgn(y r^{-1}) \tsgn(\xi |y|^{-1}))
\]
where $\tsgn$ is a mollified
signum function which satisfies
\[
\tsgn'(x)  = \frac14, \qquad |x| \leq 2
\]
Its properties are summarized in the following

\begin{lemma}
  a) $Q_r$ is bounded in $L^p$ for $1 \leq p \leq \infty$
  uniformly for $r \geq 1$.

b) $Q_r$ is also bounded in $ \tilde X_2( z)$ uniformly with respect to
$z \in \C$ and $r \geq 1$.

c) $Q_r$ satisfies the commutator estimate
\begin{equation}\label{1dv}
r^{-1} \| |D|^\frac12 \chi_r u \|_{L^2}^2 \lesssim 
  (1+|\Im z|) \|u\|_{L^2}^2 + \langle (H-z)u,Qu\rangle
\end{equation} 
\end{lemma}

\begin{proof}
a) The $L^p$ boundedness is straightforward and is left for the reader.

b) For the $\tilde X_2( z)$ boundedness we consider first the $d$
terms, which without any loss in generality we can write in the form
\[
d^{-\frac12}  \|(d^2+D^2)^\frac14  \chi_{d} u\|_{L^2} 
\]
 Since $Q$ is bounded in $L^2$ it suffices to prove the
commutator bound
\[
\| [Q_r,(d^2+D^2)^\frac14  \chi_{d}] u\|_{L^2} \lesssim \|u\|_{L^2} 
\]
But this is easily verified using the pdo calculus.

Next we consider the term 
\[
 \| (H-z) u\|_{y L^2 + \nabla L^2+ |z|^\frac12 L^2}
\]
for which it suffices  to prove the
commutator bound
\[
\| [Q_r, H] u\|_{y L^2 + \nabla L^2+ |z|^\frac12 L^2} \lesssim \|u\|_{L^2} 
\]
or equivalently, by duality,
\[
\| [Q_r, H] u\|_{L^2} \lesssim \|y u\|_{L^2} + \|\nabla u\|_{L^2} + \||z|^\frac12 u\|_{L^2} 
\]
This follows again from the pseudodifferential calculus.

c) Since $Q_r$ is skewadjoint we have the identity
\[
\langle (H-z) u, Q_r u \rangle = \langle [H,Q_r] u, u \rangle + \Im z
\langle iQ_r u, u\rangle
\]
therefore it suffices to insure that
\begin{equation}
 \langle [H,Q_r] u, u \rangle \gtrsim \| u\|_{ \bar X_2(r)}^2 + O(\|u\|_{L^2}^2).
\label{poscom}\end{equation}
For this  we compute the commutator $[H,Q]$,
\[
\begin{split}
[H,Q] &= Op^w (\{ \xi_1^2+y_1^2, \tsgn(y_1 r^{-1}) \tsgn(\xi_1
|y_1|^{-1})\}) + O(1)_{L^2 \to L^2}
\\ & = Op^w (2 r^{-1} \tsgn'(y_1 r^{-1}) \xi_1   \tsgn(\xi_1
|y_1|^{-1})) + O(1)_{L^2 \to L^2}
\end{split}
\]
Now
\[ 
r^{-1} (\chi_r^1)^2 (\xi_1^2+r^2)^\frac12 
\lesssim  r^{-1} \tsgn'(y_1 r^{-1}) \xi_1   \tsgn(\xi_1
|y_1|^{-1}) + 1 
\]
and the conclusion follows by Garding's inequality.
\end{proof}

We return to the proof of the proposition. 
  By separation of variables
the bound \eqref{1dv} extends to higher dimensions and gives
\[
r^{-1} \| |D_j|^\frac12 \chi^j_r u \|_{L^2}^2 \lesssim 
  (1+|\Im z|) \|u\|_{L^2}^2 + 2 \Re \langle (H-z)u,Q_r^j u\rangle
\]
where $Q_r^j$ the higher dimensional analogue of $Q_r$ with respect to
the $j$ variable.

The $ X_2(z)$ boundedness of $Q_r$ also extends easily to higher
dimension. Hence by Cauchy-Schwartz we obtain
\[
r^{-1} \| |D_j|^\frac12 \chi^j_r u \|_{L^2}^2 \lesssim 
  (1+|\Im z|) \|u\|_{L^2}^2 + 
  \|(H-z)u\|_{ X^*_2(z)} \| u\|_{ X_2(z)}
\]
To conclude the proof of the  estimate \eqref{eeeqa}
it remains to show that
\[
\| (H-z) u\|_{|z|^\frac12 L^2+ yL^2+ \nabla L^2} \lesssim 
\| (H-z) u\|_{X_2^*(z)} 
\]
which follows by duality from \eqref{smthe}.

\end{proof}

The final step in the $L^2$ resolvent bounds is to replace the $y'$
derivatives by angular derivatives.  Let $\nabla_\perp=
\frac{y}{|y|}\wedge \nabla$ be the angular derivative and
$|D_\perp|^\frac12$ be the corresponding fractional derivative. 

We split the coordinates into $y=(y_1,y')$ and use the notation ${}'$
for coordinates and derivatives in the obvious sense.  For $1 \le d\le
R \le \sqrt{\lambda}$ we define the sector
\[
B_{R,d} = \{ R < |y_1| < 2R, \ \  |y'| \leq d\}
\]
and $\chi_{R,d}$ a bump function in $B_{R,d}$.
Then we define the function space $X_2(\lambda,R,d)$  by
\[
\begin{split} 
 \Vert u \Vert_{X_2(\lambda,R,d)}^2 = &\ \Vert u \Vert_{L^2}^2 + 
R^{-1/2} \lambda^{-1/4} \Vert \nabla (\chi_{R,d} u) \Vert_{L^2}^2
\\ & + R^{-1/2} \lambda^{1/4} \Vert \chi_{R,d} u \Vert_{L^2}^2 
+ d^{-1/2} \Vert |D_\perp|^{\frac12} \chi_{R,d} u \Vert_{L^2}^2 
\end{split} 
\]
and $X_2^*(\lambda,R,d)$ as its dual. 

\begin{lemma} 
Suppose that $n \geq 2$ and $1 \le d \le R $. Then
\[
\|u\|_{X_2(\lambda,R,d)} \approx R^{-\frac12} \lambda^{\frac14}  \|\chi_{R,d} u\|_{L^2} + 
R^{-\frac12}\l^{-\frac14} \| \nabla \chi_{R,d} u
  \|_{L^2} + d^{-\frac12}  \||D'|^\frac12 \chi_{R,d}  u \|_{L^2} 
\]
\end{lemma}

\begin{proof} 
Within $B_{R,d}$ the angular derivatives are close to the $y'$
derivatives, namely
\[
| D_\perp u| \lesssim |D' u| + \frac{d}{R} |\nabla u|, \qquad | D' u| \lesssim |D_\perp u| + \frac{d}{R} |\nabla u|.
\]
This implies the corresponding bounds for $L^2$ norms, and the
conclusion follows by interpolation.
\end{proof} 

>From the above lemma we obtain 
\[ 
\Vert u \Vert_{X_2(\lambda,R,d)} \lesssim \Vert u 
\Vert_{\tilde X_2(\lambda,R,d) } 
\]
Hence, we may replace $\tilde X_2$ by $X_2$ in 
\eqref{l2res} and \eqref{x2lrdual}:

\begin{cor}
a) For $\lambda$ in the spectrum of $H$ we have
\begin{equation}
\| \Pi_\lambda f \|_{X_2(\lambda,R,d)} \lesssim \|f\|_{L^2}, \qquad
\| \Pi_\lambda f \|_{L^2} \lesssim \|f\|_{X_2^*(\lambda,R,d)}
\label{x2lrd}\end{equation}

b) For $z$ away from the spectrum of $H$ and $1 \leq d \leq R \lesssim
\Re z$ we have
\begin{equation}
\| (H-z)^{-1-k} f\|_{X_2(\lambda,R,d)} \lesssim
(1+|\Im z|)^{-k} \|f\|_{X_2^*(\lambda,R,d)}, \qquad k \geq 0    
\label{l2resb}\end{equation}
\end{cor}

\subsection{The $L^p$ bounds of the resolvent}

The $L^p$ bounds for the spectral projectors and the resolvent were
proved in \cite{MR1319486}, \cite{MR958904} (see also
\cite{MR1871351}).  For the sake of completeness we also present them
here in a simpler manner following the approach in \cite{MR2140267}.
We refer the reader to the same paper for further results.  We
consider pairs of exponents satisfying
\begin{equation} 
\frac2p + \frac{n}q = \frac{n}2
\label{pqscale}\end{equation}
where the range for $p$ is
\begin{equation}
 p \geq 4\ \text{ for } n=1, \qquad  p > 2\ \text{ for } n=2,
\qquad  p \geq 2\ \text{ for } n\geq 3.
\label{prange}\end{equation}
This leads to the following range\footnote{The exponent $q=\infty$ is
  actually allowed in the spectral projection bounds in dimension
  $n=2$.  However, it is not allowed in any of the resolvent bounds.}
for $q$:
\begin{equation}
 q \in [2,\infty] \text{ for } n=1, \quad   q \in [2, \infty) \text{ for } n=2,
\quad  q \in [2,\frac{2n}{n-2}] \text{ for } n\geq 3.
\label{qrange}\end{equation}
The dual exponents are denoted by $p^\prime$ and $q^\prime$
as usual.

\begin{proposition}
  Let $q$ be as in \eqref{qrange}. Then 

a) The spectral projectors
  $\Pi_\lambda$ satisfy
\begin{equation}\label{pilp}
\begin{split}
  \|  \Pi_\lambda\|_{L^{q'} \to L^2}  \lesssim 1, &\qquad 
\| \Pi_\lambda\|_{L^2 \to L^q} \lesssim 1,  \qquad
  \qquad n \geq 2
\\
  \| \Pi_\lambda\|_{L^{q'} \to L^2} 
\lesssim \lambda^{-\frac1p}, 
&\qquad \| \Pi_\lambda\|_{L^2 \to L^q} \lesssim 
\lambda^{-\frac1p}, 
 \qquad
  n  = 1
\end{split}
\end{equation}

b) For $z$ away from $n+\N$ the resolvent $(H-z)^{-1}$
satisfies
\begin{equation}\begin{split}
\|(H-z)^{-1} \|_{L^{q'} \to L^q} \lesssim (1+|\Im z|)^{\frac{1}{p}
  -\frac1{p'}},  \qquad n \geq 1
\end{split}
\label{resbd}\end{equation}
\label{plp}\end{proposition}

\begin{proof}[Outline] 
To revisit the $L^p$ bounds associated to the spectral projectors
we recall the approach in \cite{MR2140267}.  
The first step there is to
establish  pointwise bounds for the Schr\"odinger
evolution\footnote{These bounds are very robust and are in effect
  established in \cite{MR2140267} for a much larger class of operators}
\begin{equation}
\| e^{itH}\|_{L^1 \to L^\infty} \lesssim (\sin t)^{-\frac{n}2}
\end{equation}
This immediately (see also \cite{MR2094851})
leads to Strichartz estimates for the solution to the
inhomogeneous equation
\[
i v_t - H v = g, \qquad v(0) = v_0
\]
namely
\begin{equation}
\| v\|_{L^p([0,2\pi]; L^q)} \lesssim \|v_0\|_{L^2} + \|g\|_{L^{p'}([0,2\pi]; L^{q'})}
\label{sehermite}\end{equation}
where $(p,q)$ are as described in \eqref{pqscale}, \eqref{prange}.

To obtain \eqref{pilp} we apply \eqref{sehermite} to $v = e^{-i
  \lambda t} \Pi_\lambda u$, which yields $L^2 \to L^p$ bounds, and
hence by duality and selfadjointness all estimates of \eqref{pilp} for
$n \geq 2$. The case $n=1$ can be dealt with directly using the 
pointwise bounds for the Hermite functions.

We  note a  consequence of the bounds \eqref{pilp},
namely
\begin{equation}
\|(H-z)^{-1-k} \|_{L^{q'} \to L^q} \lesssim (1+|\Im z|)^{\frac{1}{p}
  -\frac1{p'}-k},  \qquad n \geq 2, \ k \geq 1
\label{resbd1}\end{equation}
which is obtained by interpolating between $q=2$ and $q =
\frac{2n}{n-2}$.

Similarly  we get
\begin{equation}
  \|(H-z)^{-1} \|_{L^{q'} \to L^2} \lesssim (1+|\Im z|)^{-\frac1{p'}},
  \qquad n \geq 2.
\label{resbd2}\end{equation}

Then  we apply \eqref{sehermite} to 
\[
v(x,t) = \chi(t) e^{-i z t} u(x), \qquad 
g = \chi'(t) e^{-izt}  u(x) + \chi(t) e^{-i z t} (H-z) u
\]
where $\chi$ is a unit bump function on an interval of size $(1+|\Im z|)^{-1}$.
This yields
\[
\begin{split}
\| u\|_{L^q} \lesssim (1+|\Im z|)^{\frac1p} \|u\|_{L^2} + (1+|\Im
z|)^{\frac{1}{p} -\frac1{p'}}
\|(H-z)
u\|_{L^{q'}} \lesssim  \|(H-z)u\|_{L^{q'}}.
\end{split}
\]
concluding the proof of \eqref{resbd} for $n \geq 2$. The case
$n=1$ is a variation on the same theme.
\end{proof}

\subsection{Combining the estimates}

Here we  combine the $L^2$ and the $L^p$ components in the 
resolvent bounds:

\begin{proposition}
 For $z$ away from $n+2\N$ the resolvent $(H-z)^{-1}$
satisfies
\begin{equation}
\|(H-z)^{-1} \|_{L^{q'} \to X_2(\Re z,R,d)} \lesssim (1+|\Im z|)^{\frac{1}{p}
  -\frac1{2}},  \qquad n \geq 2, \qquad (n,q) \neq (2,\infty)
\label{resbdz}\end{equation}
with the obvious modification for $n=1$.
\label{pl2p}\end{proposition}

\begin{proof}[Proof of Proposition~\ref{pl2p}]
Taking into account the bounds   \eqref{resbd} and \eqref{resbd2},
it remains to prove the estimate
\[
\begin{split} 
\|u\|_{ \tilde X_2(\Re z,R,d)} \lesssim & (1+|\Im z|)^{\frac12-\frac{1}{p'}}
\| (H-z) u\|_{L^{q'}} \\ &  + (1+|\Im z|)^{\frac12-\frac1p} \|u\|_{L^q}
+ (1+|\Im z|)^\frac12 \| u\|_{L^2}
\end{split} 
\]
But this follows from \eqref{poscom} in the same way as for
Proposition \ref{pl2res} since the operator $Q$ is bounded in $L^q$.
\end{proof}

\section{$L^p$ estimates in the flat case and parametrix bounds}
\label{slpflat}
In this section we begin with the mixed norm $L^pL^q $ Carleman
estimates in the simplest case, i.e. with constant coefficients and a
polynomial weight.  These were proved in \cite{MR1769727} except for 
the endpoint which was obtained later in \cite{MR1871351} .

After a conformal change of coordinates and conjugation with respect
to the exponential weight the Carleman estimates reduce to proving
$L^pL^q $ estimates for a parametrix $K$ for $\partial_t -H +\tau$. In
this article we need a stronger version of these bounds, where we 
add in localized $L^2$ norms.

In a simplified form, Escauriaza-Vega's result in \cite{MR1871351}  has the form:

\begin{theorem}\cite{MR1871351} Let $p$ and $q$ be as above. Then
\[
\| t^{-\tau} e^{-\frac{x^2}{8t}} u\|_{L^\infty(L^2)\cap L^p(L^q)}
\leq \| t^{-\tau} e^{-\frac{x^2}{8t}} (\d_t + \Delta ) u\|_{L^1(L^2)+L^{p'}(L^{q'})}, 
\]
 for all $u$ with compact support in $\R^n \times [0,\infty)$ 
 vanishing of infinite order at $(0,0)$ uniformly
with respect to $4\tau$ with a positive distance  from integers.
\label{plpflat} \end{theorem}

One can write the estimate in the $(s,y)$ coordinates
using the same transformation as in Section~\ref{sl2flat}:
\begin{equation}
\| e^{\tau s} v\|_{L^\infty(L^2)\cap L^p(L^q)} \lesssim \| e^{\tau s}
(\partial_s +H) v\|_{L^1(L^2)+L^{p'}(L^{q'})}
\end{equation}
Setting $w = e^{\tau s} v$ this becomes
\begin{equation}
\| w\|_{L^\infty(L^2)\cap L^p(L^q)} \lesssim \| 
(\partial_s + H- \tau) w\|_{L^1(L^2)+L^{p'}(L^{q'})}
\label{flatlph}\end{equation}

Denoting by $\Pi_\lambda$ the spectral projection onto the $\lambda$
eigenspace of $H$ we obtain a parametrix $K$ for $(\partial_t - H +\tau)$,
\[
K(\partial_t + H -\tau) = I
\]
where the $s$-translation invariant kernel of $K$ is 
\[
K(s) = \sum_{\lambda \in \N} \Pi_{\lambda} e^{s(\tau-\lambda)}
1_{s(\tau-\lambda) < 0}
\]
Since $w$ decays at $\pm \infty$ we have 
\[
w = K (\partial_s + H- \tau) w
\]
therefore \eqref{flatlph} can be rewritten in the form 
\begin{equation}
\| Kf \|_{L^\infty(L^2)\cap L^p(L^q)} \lesssim \| f \|_{L^1(L^2)+L^{p'}(L^{q'})}
\end{equation}

The main result of this section is an improvement of 
\eqref{flatlph}, namely

\begin{proposition}
Assume that $\tau$ is away from $n+\N$ and that
\[
1 \leq d \leq R \lesssim \tau
\]
Then
\begin{equation}
\| Kf \|_{L^\infty(L^2)\cap L^p(L^q)\cap L^2 X_2(\tau,R,d)} \lesssim 
\| f \|_{L^1(L^2)+L^{p'}(L^{q'})+L^2 X_2^*(\tau,R,d)}
\label{flatlp2h}\end{equation}

\end{proposition}

\begin{proof}
  We work in dimension $n \geq 2$; some obvious adjustments are needed
  in dimension $n=1$, which is slightly easier.   We consider four
  endpoints:

\noindent
{\bf A: The $L^1L^2 \to L^\infty L^2$ bound} follows easily since
the projectors $\Pi_\lambda$ are $L^2$ bounded.

\medskip 
\noindent
 {\bf B: The $L^1L^2 \to L^p L^q $ bound.}
Here it suffices to prove 
\[ \|K(.)f \|_{L^p_tL^q_x} \lesssim \| f \|_{L^2} \]
Splitting $f$ into spectral projections and using \eqref{pilp}
we obtain
\[
\| K(t) f\|_{L^q} 
\lesssim \sum_{\lambda}  e^{-|(\lambda -\tau)t|} \|\Pi_\lambda f\|_{L^2}
\]
For $|t| \geq 1$ we can use Cauchy-Schwartz to obtain
\[
\| K(t) f\|_{L^q \cap X_2(R,d)} \lesssim e^{-c|t|} \|f\|_{L^2}
\]
which suffices for all $q$. For $|t| \leq 1$ we consider the most difficult
case $p =2$  and compute
\[
\begin{split}
  \| K(t) f\|_{L^2([-1,1],L^q)}^2 \hspace{-25pt} & \hspace{25pt} \lesssim \int_{-1}^1 \left (\sum
    e^{-|(\lambda -\tau)t|} \|\Pi_\lambda f\|_{L^2} \right)^2 \\ =
  &\sum_{\lambda,\mu} \frac{1}{|\lambda -\tau|+|\mu-\tau|}
  \|\Pi_\lambda f\|_{L^2} \|\Pi_\mu f\|_{L^2} \\ \approx &\sum_{0 \leq
    i, j} 2^{-i-j} \left(\sum_{|\lambda-\tau| \approx 2^i}
    \|\Pi_\lambda f\|_{L^2}\right)\left( \sum_{|\mu-\tau|
      \approx 2^{i+j}} \|\Pi_\mu f\|_{L^2}\right) \\ \lesssim& \sum_{0
    \leq j} 2^{-\frac{j}2} \sum_{0 \leq i}
  \left(\sum_{|\lambda-\tau| \approx 2^i} \|\Pi_\lambda
    f\|_{L^2}^2\right)^\frac12 \left(\sum_{|\mu-\tau| \approx
      2^{i+j}} \|\Pi_\mu f\|_{L^2}^2\right)^\frac12
\\ \lesssim & \|f\|_{L^2}^2
\end{split}
\]

\smallskip 

\noindent{\bf C: The $L^1L^2 \to X_2(\tau,R,d)$ bound} for $K$ follows
in the same way from
\[ \Vert \Pi u \Vert_{X_2(\tau,R,d)} \lesssim   \Vert u \Vert_{L^2}. \]

\smallskip 

\noindent{\bf D: The $L^{p'} L^{q'} + L^2 X_2^*(\tau,R,d) \to L^\infty
  L^2$ bound} for $K$ is equivalent to the $L^1L^2 \to L^p L^q\cap
X_2(\tau,R,d)$ bound for $K^*$.  By reversing time this is seen to be
the same as the $L^1L^2 \to L^p L^q$ bound for $K$.

\smallskip 

\noindent{\bf E: The $L^{p'} L^{q'} + L^2(X_2^*(R,d)) \to L^p L^q \cap L^2
  X_2(R,d)$  bound}. Using \eqref{pilp} and \eqref{x2lrd} directly
yields
 \[ 
\Vert K(s) \Vert_{L^{\frac{2n}{n+2}}+X_2^*(R,d)\to
  L^\frac{2n}{n-2}\cap  X_2(R,d)} \lesssim 
  \sum_{\lambda \in \N} e^{|s(\tau-\lambda)|}
\lesssim s^{-1} e^{-cs} 
\]
Similarly we obtain
 \[ 
\Vert K(s) \Vert_{L^2 \to
  L^\frac{2n}{n-2}\cap  X_2(R,d)} 
\lesssim s^{-\frac12} e^{-cs}, \qquad \Vert K(s) \Vert_{L^{\frac{2n}{n+2}+X_2^*(R,d)}\to
 L^2} \lesssim  s^{-\frac12} e^{-cs}
\]

Interpolation with the $L^2$ estimate gives
\[  
\Vert K(s) \Vert_{L^{q^\prime}\to L^q} 
\lesssim  s^{-\frac{2}p}
\]
and
\[
 \Vert K(s) \Vert_{L^{q^\prime}\to X_2(R,d)} 
\lesssim  s^{-\frac{1}p-\frac12}, \qquad \Vert K(s) \Vert_{ X_2^*(R,d)\to L^q} 
\lesssim s^{-\frac{1}p-\frac12}.
\]
If $p > 2$ then the Hardy-Littlewood Sobolev inequality implies 
\[   
\Vert K* f \Vert_{L^pL^q} \lesssim \Vert f \Vert_{L^{p'}L^{q'}}. 
\]
and
\[
\Vert K* f \Vert_{L^2 X_2(R,d)} \lesssim \Vert f \Vert_{L^{p'}L^{q'}},
\qquad \Vert K* f \Vert_{L^pL^q} \lesssim \Vert f \Vert_{L^{2} X_2^*(R,d)}. 
\]
With obvious changes the analysis is similar if $n=1,2$.

It remains to prove the $L^2 \to L^2$ type bounds, namely
\[
\| Kf\|_{L^2 X_2(R,d)} \lesssim  \Vert f \Vert_{L^{2} X_2^*(R,d)}
\qquad (n=1,2)
\]
respectively 
\[
\| Kf\|_{L^2 (X_2(R,d)\cap L^{\frac{2n}{n-2}})} \lesssim  
\Vert f \Vert_{L^{2} (X_2^*(R,d)+L^{\frac{2n}{n+2}})}
\qquad (n>2)
\]

For this, following an idea in \cite{MR1871351}, we consider a dyadic
frequency decomposition in time.  By the Littlewood-Paley theory it
suffices to prove the bound for a single dyadic piece at frequency
$2^j$, namely
\begin{equation}
\|S_j(D_s) Kf\|_{L^2 (X_2(R,d)\cap L^{\frac{2n}{n-2}})} \lesssim  
\Vert f \Vert_{L^{2} (X_2^*(R,d)+L^{\frac{2n}{n+2}})}
\qquad (n>2)
\label{plks}\end{equation}
and its one and two dimensional counterpart.

Taking a time Fourier transform we can write (for $f \in \mathcal{S}(\R^n)$)
\[
\widehat{S_j K}(\sigma) f = s(2^{-j} \sigma) \sum_\lambda
\frac{1}{\lambda -\tau -i \sigma} \Pi_\lambda f  = s(2^{-j} \sigma)
(H -\tau - i\sigma)^{-1} f
\]
therefore by the inversion formula
\[
\begin{split}
(S_j K)(t) f = & \int e^{it \sigma} s(2^{-j} \sigma)(H -\tau -
i\sigma)^{-1} f d\sigma \\=&- t^{-2} 
 \int e^{it \sigma} \frac{d^2}{d\sigma^2} (s(2^{-j} \sigma)(H -\tau -
i\sigma)^{-1}) f d\sigma
\end{split}
\]
Hence using the resolvent bounds \eqref{resbd} and \eqref{resbd1}
and the first line for $|t|\le 2^{-j}$ and the second line for $|t|\ge 2^{-j}$
we obtain
\[
\| S_j K(t)\|_{L^{\frac{2n}{n+2} + X_2^*(R,d)} \to
  L^\frac{2n}{n-2}\cap X_2(R,d)} 
\lesssim \frac{2^{j}}{1+2^{2j} t^2}
\]
and the similar estimate in one and two dimensions.  The bound on the right is
integrable in $t$, therefore \eqref{plks} follows.
\end{proof}

\section{ Modified weights and pseudoconvexity}
\label{convex}

 The main result of this section, Theorem \ref{core}  is a considerable improvement of Section 
\ref{sl2flat}. 
The weights $t^{-\tau}$ in Section \ref{sl2flat}, while easy to use,
satisfy merely a degenerate pseudoconvexity condition, in the sense
that the selfadjoint and the skewadjoint parts of the operator in
\eqref{ttau} commute. This is in contrast to strong pseudoconvexity
where one obtains better $L^2$ bounds from the positivity of the
commutator.  A perturbation argument easily implies an $L^2$ Carleman
estimate for variable coefficients as soon as $ g = \id + O(t)$.
However, even arbitrarily small perturbations of $g$ from $\id$ at
$t=0$ destroy the pseudoconvexity.

To obtain results for general variable coefficients we need a more
robust weight with additional convexity. A good way of doing this is
by adding convexity in $t$ and by using a weight of the form
\[
e^{h(-\ln t)}
\]
with a convex function $h$. Then we obtain for the heat operator the
strengthened $L^2$ estimates of Proposition \ref{bell}.  The
assumption of vanishing of infinite order forces us to work with
functions $h$ with at most linear growth at infinity.  This in turn
limits the convexity of $h$, and hence the gain from the convexity in
the $L^2$ bounds.

These Carleman inequalities with the weight $e^{-h(\ln t)}$ are more
stable with respect to perturbations.  They can be obtained for
coefficients satisfying
\begin{equation} \label{hold} |g-\id| + (t+ |x|^2) |\d_t g | +
  (t+|x|^2)^{1/2} |\nabla g| \lesssim (t+|x|^2)^\varepsilon.
\end{equation} 
with suitable functions $h$.  It is not difficult to weaken
\eqref{hold} almost to our conditions \eqref{g} and \eqref{g1}. This
venue was pursued by Escauriaza and Fern\'andez \cite{MR1971939}.

In this paper we seek to obtain $L^p$ Carleman inequalities and also
to handle $L^p$ gradient potentials. Both require good spatial and
temporal localization, which depends on the strength of the $L^2$
estimates.  The weights $e^{h(-\ln t)}$ seem to be insufficient for
this purpose.

Consequently we consider a larger class of weights of the form
\[
e^{h(-\ln t) + \phi( x t^{-1/2}, -\ln(t) ) }
\]
having some additional convexity in $y=x t^{-1/2}$.  Here we think of
$\phi$ essentially as a function of $y$ with a milder dependence on
$s=-\ln t$. Obtaining pseudoconvexity is not entirely straightforward
because the Hamilton flow for the Hermite operator $H$ is periodic so
no nonconstant function of $y$ can be convex along its orbits.  We
note that the projection of the orbits to the $y$ space are ellipses
of size $O( \sqrt{\tau})$ where $\tau$ is the energy, centered at $0$.
Hence we can choose $\phi$ to be convex in $y$ for $|y| \ll \sqrt
\tau$. We compensate the lack of convexity of $ \phi$ when $|y|
\approx \sqrt \tau$ by the $s$ convexity of $h$.  To elaborate this
idea we explain the precise setup.

Let $\delta_1$ be small positive constant.  We begin with constants
$\{\alpha_{ij}\}_{\mathcal{A}}$ (see \eqref{matA} and \eqref{matAtau}
for the notation) which control the regularity of the
coefficients\footnote{denoted generically by $d$ here and later}
$g^{kl}-\delta^{kl}$, $d^{kl}$ and $e^{kl}$ of $P$ given by
\eqref{pgen} as in \eqref{g}.
\begin{equation}\label{gbound}
  \delta_1 \alpha_{ij} =  \Vert d \Vert_{L^\infty(A_{ij})} + 
  e^{j-2i} \Vert  d \Vert_{\lip_x(A_{ij})} +  \|d\|_{C^{m_{ij}}_t(A_{ij})}.
\end{equation} 
The condition \eqref{g}  guarantee that for all $\tau \ge
1$
\begin{equation} \label{alphaij} \Vert \alpha_{ij}
  \Vert_{l^1(\mathcal{A}(\tau))} \le 1. \qquad \Vert \alpha_{ij}
  \Vert_{l^\infty(\mathcal{A})} \le 1.
\end{equation}

We first adjust the $\alpha_{ij}$'s upward so that they vary slowly
and do not concentrate in irrelevant regions. This readjustment
depends on the choice of the parameter $\tau$.

\begin{lemma}
  Let $\alpha_{ij}$ be a sequence satisfying \eqref{alphaij}.  Then
  for each $\tau \gg 1$ there exists a double sequence
  $(\e_{ij})_{\mathcal{A}(\tau)}$ with the following properties:
  \begin{enumerate}
  \item \label{e1} For each $(i,j) \in \mathcal{A}(\tau)$
    \[
    \alpha_{ij} \le \e_{ij}
    \]
  \item \label{e2} We have $\e_{ij} \in l^1(\mathcal{A}(\tau))$,
    \[
    \Vert \e_{ij} \Vert_{l^1( \mathcal{A}(\tau))} \lesssim 1 .
    \]
  \item \label{e3} The sequence $\e_{ij}$ is slowly varying,
    \[
    |\ln \e_{i_1j_1} - \ln \e_{i_2j_2}| \le \frac12 ( |i_1-i_2| +
    |j_1-j_2|), \qquad (i_1,j_1), (i_2,j_2) \in \mathcal{A}(\tau).
    \]
  \item \label{e4} The sequence $(\e_i)$ defined by
    \[
    \e_i = \sum_{j: (i,j) \in \mathcal{A}(\tau) } \e_{ij}, \qquad i
    \ge \ln \tau.
    \]
    has the following properties
    \begin{equation}\label{locbound}
      |\ln \e_{i_1} - \ln \e_{i_2}| \le \frac12  |i_2-i_1|,\quad   
      \e_{ij} \lesssim \e_i, \quad \e_{i  [ \ln \tau/2]+2} \approx \e_i
    \end{equation}
  \item \label{e5} For each $i\ge \ln \tau$ there exists an unique $ 0
    \le j(i) \le [ \ln \tau/2]+2$ with the following properties:
    \begin{equation} \label{ij0} \e_{ij(i)} \approx \e_i.
    \end{equation} 
    and
    \begin{equation} \label{jnot}
      \begin{split} 
        \e_{ij} \leq  e^{-j} \tau^{-1/2}  & \qquad \text{ if } 0 \le j \leq  j(i),  j(i) \ne 0  \\
        \e_{ij} > e^{-j} \tau^{-1/2} & \qquad \text{ if } j(i) < j
        \leq [ \ln \tau/2]+2.
      \end{split}
    \end{equation} 
  \end{enumerate}
  \label{lepsilon}\end{lemma}

We shall see that $j(i)$ is an important threshold.  If $j \ge j(i)$
then we can localize our estimates to the corresponding $A_{ij}$ and
even to smaller sets. On the other hand, we cannot localize to sets
smaller than
\begin{equation}
  B_{i0} = \bigcup_{j \le j(i)} A_{ij}.  
\end{equation}

\begin{proof}
  To fulfill the conditions (\ref{e1})-(\ref{e4}) we simply mollify
  the $\alpha_{ij}$,
  \[
  \tilde \e_{ij} = \max_{(k,l)\in \mathcal{A}(\tau)} \alpha_{kl}
  e^{-\frac12(|i-k|+|j-l|)}, \quad \tilde \e_i = \sum_{j=0}^{[\ln
    \tau/2]+2} \tilde \e_{ij}.
  \]
For the last part of \eqref{e4} we redefine 
\[
\tilde{\e}_{ij} := \tilde{\e}_{ij} + e^{-\frac12 |j-([\ln \tau/2]
  +2)|}  \tilde{\e}_i
\]
This also increases $ \tilde{\e}_i$ by a fixed factor.

  For \eqref{e5} we begin with a preliminary guess for $j(i)$ which we
  call $j_0(i) \in \mathbb{R}^+$.  We consider three cases.
  \[ j_0(i)= \left\{
    \begin{array}{clc}
      \ln \tau/2  & \text{ if }&  \tilde \e_i < \tau^{-1}, \\ 
      - \ln  ( \tilde\e_i \tau^{1/2}) & \text{ if } & \tau^{-1} \le  \tilde\e_i < \tau^{-1/2}, \\
      0 & \text{ if }&  \tau^{-1/2} \le  \tilde\e_i. 
    \end{array} 
  \right.
  \]
  We define
  \begin{equation}\label{eijnew}
    \e_{ij}:= \max\{ \tilde\e_{ij}, 2e^{-\frac12|j -j_0(i)|} \tilde\e_i \}, 
  \end{equation} 
  which is still slowly varying because $\tilde\e_i$ is slowly
  varying.

  We define $j(i)$ according to \eqref{jnot}.  It is uniquely
  determined since the sequence $\e_{ij}$ is slowly varying compared
  to $e^{-j}\tau^{\frac12}$.  Since $\tilde \e_{ij}$ is slowly varying
  we must have $\tilde\e_{ij} \le \tilde \e_i/2$. This allows us to
  conclude that for $j$ close to $j_0(i)$, the second term in
  \eqref{eijnew} is larger than the first one,
  \[ \e_{ij} = 2e^{-\frac12|j -j_0(i)|} \tilde\e_i\qquad \text{ for }
  |j-j_0(i)| \le 2. \] If $j_0(i)=0$ then $ \e_{i0}= 2 \tilde \e_i \ge
  2 \tau^{-1/2}$ and hence $j(i)=0$. If $0 < j_0(i) < \ln\tau/2$ then
  for $|j-j_0(i)|\le 2$ we have
  \[ \e_{ij} = 2e^{-\frac12|j -j_0(i)|} e^{-j_0(i)} \tau^{-1/2} \]
  therefore $j_0(i) -2 < j(i) \le j_0(i)$.  If $j_0(i) = \ln \tau /2$
  then for $|j-j_0(i)|\le 2$ we have
  \[ \e_{ij} \le 2e^{-\frac12|j -j_0(i)|} e^{-j_0(i)} \tau^{-1/2} \]
  and we arrive again at $j_0(i)-2 \le j(i)$.  In all three cases we
  have $|j_0(i)-j(i)|\le 2$ therefore \eqref{ij0} holds. We observe
  that $\e_i\lesssim \tilde \e_i$.
\end{proof}

The sequence $(\e_{ij})_{\mathcal{A}(\tau)}$ is used to describe the
amount of spatial convexity needed in the region $A_{ij}$, which will
be reflected in the construction of $\phi$ below.  The partial sums
$\e_i$ measure the amount of $s$-convexity needed in $ [i,i+1]$.  The
purpose of part \eqref{e5} above is to correlate the two amounts in a
region where they have the same strength (where $j$ is close to
$j(i)$).

Our weights have the form
\begin{equation}
  \psi(s,y) = h(s) + \phi(s,y)
  \label{psi}\end{equation}
Their choice is described in the next two lemmas:

\begin{lemma}
  Let $\tau$ and $(\e_{i})$ be as in Lemma~\ref{lepsilon}.  Then there
  is a convex function $h$ with the following properties:
  \begin{enumerate}
  \item $h' \in [\tau,2\tau]$.
  \item $ h''(s) + dist(h'(s),\N) > \frac14 $.
  \item $\e_i \tau \lesssim h''(s) \lesssim \e_i \tau+1$ for $s \in
    [i,i+1]$.
  \item $|h'''| \lesssim h''$.
  \end{enumerate}
  \label{h}\end{lemma}

The proof of the lemma is fairly straightforward and uses only the
fact that $(\e_i)$ is slowly varying and summable. The second part is
needed in order to avoid the eigenvalues of the Hermite operator.

\begin{lemma}
  Let $\tau$, $(\e_{ij})$ and $(\e_i)$ be as in Lemma~\ref{lepsilon}.
  Then there exists a smooth spherically symmetric function
  \[
  \phi: \R \times \R^n \to \R
  \]
  with the following properties:
  \begin{enumerate}
  \item (Bounds) The function $\phi$ is supported in $|y|\le 2
    \tau^{1/2}$ and satisfies
    \begin{equation} \label{phitime} 0 \le \phi(s,y) \lesssim \e_i
      \tau, \qquad |\d_s \phi(s,y))|+|\d_s^2 \phi(s,y)| \lesssim \e_i
      \tau
    \end{equation}
    and
    \begin{equation} \label{phibound} \sum_{l,k=0}^3 (1+|y|)^{k}
      |D_y^{1+k} \d_s^l \phi| \lesssim \epsilon_i \tau^{1/2} \qquad
      \text{ for } i \le s \le i+1,
    \end{equation}
  \item (Monotonicity)
    \begin{equation} \label{phimon} \d_r \phi (s,y) \approx \e_i
      \tau^\frac12 \quad \text{ for } (s,y) \in A_{ij}, \quad (i,j)
      \in \mathcal{A}(\tau), \quad j \ge j(i)+1
    \end{equation}
  \item (Convexity)
    \begin{equation} \label{phicon} (1+|y|) \partial^2_r \phi(s,y)
      \approx \e_{ij} \tau^{\frac12} \qquad \text{ in } A_{ij},\quad
      (i,j) \in \mathcal{A}(\tau).
    \end{equation} 
  \end{enumerate}
  \label{phi}\end{lemma}

\begin{proof}
 
  Let
  \[
  \phi_j(y) = \sqrt{e^{2j} + |y|^2}, \qquad j \geq 0.
  \]
  We fix a smooth partition of unity $1 = \sum \eta(s-i)$ and define
  \[
  \ln a_j(s) =\sum_i \eta(s-i) \ln \e_{ij}.
  \]
  These functions satisfy the bounds
  \begin{equation}\label{atime}
    a_j(s) \approx  \varepsilon_{ij}, \quad i \le s \le i+1, 
    \qquad | a_j^\prime|, |a_j^{\prime \prime}|, |a'''_j| \lesssim    a_j.  
  \end{equation} 
  Their sum satisfies
  \[
  a(s):= \sum_{j=0}^{[\ln \tau/2]+2} a_j(s) \approx \e_i, \quad i\le s
  \le i+1.
  \]
  We define
  \[
  \phi(s,y) = \tau^\frac12 \chi(|y|\tau^{-1/2} ) \sum_{j=0}^{[\ln
    \tau/2]+2} a_j(s) \phi_j(|y|)
  \]
  where $\chi$ is a smooth function supported in $[-2,2]$ and
  identically $1$ in $[-\frac32,\frac32]$.  We verify the properties:
  \[ 0 \le \phi(s,y) \lesssim a(s) \tau \lesssim \e_i \tau, \quad i
  \le s \le i+1. \] The remaining part of \eqref{phitime} follows from
  \eqref{atime}.  Estimate \eqref{phibound} is a consequence of
  \eqref{atime} and
  \[ |(1+|y|)^k D^{k+1} \phi_j(y)| \lesssim 1, \quad 0 \le k \le 3. \]
  The upper bound from \eqref{phimon} is covered by \eqref{phibound}
  and the lower one follows from
  \[ \d_r \phi(s,y) \gtrsim \tau^{1/2} \sum_{j \le j(i)} a_j \approx
  \tau^{1/2} \sum_{j=0}^{j(i)} \e_{ij} \sim \e_i \tau^{1/2}
  \]
  in $A_{ij}$ with $j \ge j(i)$ where we use $\e_{ij(i)} \approx \e_i$.
  The assertion \eqref{phicon} follows from immediate bounds on second
  derivatives of the $\phi_j$.
\end{proof}

Our aim in this section is to prove $L^2$ Carleman estimates for the
variable coefficient operator $P$ with the exponential weight
\[ e^{\psi(-\frac{\ln t}4, \frac{2x} {t^{1/2}})}e^{-\frac{|x|^2}{8t}}.
\] Here
\[ \psi(s,y) = h(s) + \delta_2 \phi(s,y). \] where $\delta_2$ is a
small constant and $h$ and $\phi$ are as in in Lemma \ref{h} and
\ref{phi}.

The calculations are involved.  For a first orientation we outline the
key part of the argument for the constant coefficient heat equation.
Using the change of coordinates of Section \ref{sl2flat} we transform
the problem to weighted estimates for the operator $ P_0= \d_s + H$
and the exponential weight $e^{\psi(s,y)}$.  This translates to
obtaining bounds from below for the conjugated operator
\[
\dtpsio = e^{\psi(s,y)} P_0 e^{-\psi(s,y)}.
\]

\begin{lemma}
  Let $\tau$ be large enough.  Let $\psi$ be as in \eqref{psi} with
  $h,\phi$ as in the above two Lemmas~\ref{h},\ref{phi} with $\delta_2
  \ll 1$ .  Then the operator $\dtpsio$ satisfies the bound from below
  \[
  \|(h'')^\frac12 v \|^2 + \delta_2\tau^{-1} ( \|a_{int}^2 \nabla
  v\|^2 + \| a_{\perp}^2 \nabla_{\perp} v\|^2) \lesssim \|\dtpsio
  v\|_{L^2}^2
  \]
  for all functions $v$ supported in $\{|y|^2 \leq 9 \tau\}$ where the
  weights $a_{int}$, $a_\perp$ are defined by
  \[
  a_{int}^4 = \e_{ij} (1+|y|)^{-1} \tau^\frac32, \quad a_{\perp}^4 =
  1_{j \ge j(i)} \e_{i} (1+|y|)^{-1} \tau^\frac32 \qquad \text{in }
  A_{ij}.
  \]
\end{lemma}
\begin{proof}
  We decompose $\dtpsio$ into its selfadjoint and its skewadjoint part
  \[
  \dtpsio = L_{0,\psi}^r + L_{0,\psi}^i
  \]
  where
  \begin{equation}
    L_{0,\psi}^r := -\Delta_y + y^2 - \psi_s - \psi_y^2, \qquad
    L_{0,\psi}^i := \d_s + \psi_y  \d + \d \psi_y.
    \label{lopsi}\end{equation}
  Expansion of the norm gives
  \begin{equation}
    \|(L_{0.\psi}^r + L_{0,\psi}^i) v\|^2_{L^2} = \|L_{0,\psi}^r  v\|^2_{L^2} + \|L_{0,\psi}^i  v\|^2_{L^2}
    + \langle [L_{0,\psi}^r,L_{0,\psi}^i]v,v\rangle
    \label{cep1}\end{equation}
  The conclusion of the lemma follows from the commutator bound
  \begin{equation}
    \langle [L^r_{0,\psi},L^i_{0,\psi} ] v,v\rangle \gtrsim \|
    (h'')^\frac12 v \|^2 + \delta_2\tau^{-1} ( \|a_{int}^2 \nabla v\|^2 + \|
    a_{\perp}^2 \nabla_{\perp} v\|^2)
    \label{com0}\end{equation}

  The commutator is explicitly computed
  \[ [L_{0,\psi}^r,L_{0,\psi}^i] = \psi_{ss} + 4 \psi_y\psi_{yy}
  \psi_y - 4 \d \psi_{yy} \d - 4 y \psi_y + 4\psi_{y} \psi_{sy} -
  \Delta^2 \psi
  \]
  Since $\delta_2 \ll1$ the first term has size $h''(s)$.  The second
  one is nonnegative since $\psi$ is convex for $|y|^2 < 9\tau$.

  The Hessian of the radial function $\psi$ can be written in the form
  \begin{equation}
    \psi_{yy} = \psi_{rr} \frac{y}{|y|}
    \otimes \frac{y}{|y|} + \frac{\psi_r}{r} \left( \id -\frac{y}{|y|}
      \otimes \frac{y}{|y|}\right)
    \label{psirr}\end{equation}
  One can see that the radial and angular derivatives carry different
  weights.  Our construction of $\phi$ guarantees that
  \[
  \psi_{rr} \lesssim \frac{\psi_r}{r}, \qquad \psi_{yy} \gtrsim
  \psi_{rr} \id
  \]
  hence the weight $\psi_{rr}$ can be used for all derivatives.  For
  the size of the two weights we have
  \[
  \psi_{rr} \approx a_{int}^4, \qquad \frac{\psi_r}{r} \gtrsim
  a_\perp^4
  \]
  This gives the last two terms in \eqref{com0}.

  It remains to see that the remaining terms in the commutator are
  negligible compared to the first term on the right hand side of
  \eqref{com0}.  For this we use the bound \eqref{phibound} to
  conclude that in $A_{ij}$ we have
  \[
  | - 4 y \psi_y + 4\psi_{y} \psi_{sy} - \Delta^2 \psi| \lesssim
  \delta_2 \e_i \tau \lesssim \delta_2 h''
  \]
\end{proof}

To switch to operators with variable coefficients it is convenient to
extend the weights to the full space and to regularize them.
Precisely we shall assume that
\begin{equation}
\begin{split}
  a_{int}^4(s,y) &\approx  \e_i \tau
 \qquad  \qquad \qquad   \text{ in }A_{ij} \text{ if } |y|^2 \ge \tau
  \\
  a_{int}^4(s,y) &\approx \e_{ij} (1+|y|)^{-1}  \tau^\frac32
  \quad   \text{ in }A_{ij} \text{ if } |y|^2 \le \tau. 
\end{split}
\end{equation}
Observe that the two cases above match since $\e_i \approx \e_{ij}$ in
the region where $ y^2 \approx \tau$. We also introduce a modification
$a$ of $a_{int}$ which is used to include the effect of the spectral
gap in regions where we have very little convexity:
\begin{equation}
\begin{split}
  a^4(s,y) &\approx 1+ \e_i \tau
 \qquad  \qquad \qquad   \text{ in }A_{ij} \text{ if } |y|^2 \ge \tau
  \\
  a^4(s,y) &\approx 1+ \e_{ij} (1+|y|)^{-1}  \tau^\frac32
  \quad   \text{ in }A_{ij} \text{ if } |y|^2 \le \tau. 
\end{split}
\end{equation}
Finally we choose $a_\perp$ with the properties
\begin{equation}
\begin{split}
  \supp a_\perp \subset&  \bigcup \{ A_{ij}:\  j(i) -1 \leq j \leq
  \frac12 \ln \tau \}  
\\
  a^4_\perp(s,y) \lesssim &  \e_i (1+|y|)^{-1} \tau^\frac32
  \quad  \text{ in }   A_{ij}
  \\
  a^4_\perp(s,y) \approx & \e_{i} (1+|y|)^{-1}  \tau^\frac32
  \quad   \text{in } A_{ij} \text{ if }  j(i)  \leq j \leq
  \frac12 \ln \tau -1
\end{split}
\end{equation}

The bounds for the weights from above are assumed to remain true after
applying powers of the differential operators $\partial_s$,
$\partial_y$ and $ y \partial_y$ to them.

Consider now a the more general class of operators $P$ with real
variable coefficients given by \eqref{pgen}.  We repeat the change of
coordinates and write in the $(s,y)$ coordinates:
\[
4 e^{-4s} P = - \frac{\d}{\d s} - 2 y\frac{\d}{\d y} + \d_i g^{ij}
\d_j + y_i d^{ij} \d_j + \d_i d^{ij} y_j + y_i e^{ij} y_j 
\]
This further leads to
\[
4e^{-(n+4)s-\frac{y^2}{2}} P e^{ ns+\frac{y^2}{2}}= - \tilde P
\]
where $\tilde P$ is given by
\[
\begin{split}
  \tilde P = & \ \frac{\d}{\d s} - \partial_i g^{ij} \partial_j -
 y_i  (g^{ij} - 2\delta^{ij}+2d^{ij}+e^{ij}) y_j \\ & 
- y_i(g^{ij} -\delta^{ij}+d^{ij})
  \partial_j - \d_i(g^{ij} -\delta^{ij}) y_j
\end{split}
\]
We rewrite it in the generic form
\[
 \tilde P =  P_0 - \d d\d -
 y d y - y d \d -
  \d d y
\]
with $P_0 = \partial_s+H$.

To simplify as much as possible the proof of the main $L^2$ Carleman
estimate we introduce a stronger condition on the regularity of the
coefficients:
\begin{equation}
  \begin{split} 
    |d| + \jy (|d_y| + \tau^{-\frac12} |d_{yy}| + \tau^{-1}
    |d_{yyy}| + \tau^{-\frac12} |d_s| ) & \lesssim \delta_1 \e_{ij}
    \qquad \text{in } A_{ij}
    \\
   |d| + \jy (|d_y| + \tau^{-\frac12} |d_{yy}| + \tau^{-1}
    |d_{yyy}| + \tau^{-\frac12} |d_s| )  &\lesssim \delta_1
  \end{split}
  \label{regg}\end{equation}
This improved regularity will be gained later on by regularizing the
coefficients.  We are now in the position to formulate the Carleman
estimate.

\begin{proposition}
  Let $\tau$ be large enough and $\delta_1 \ll \delta_2 \ll 1$.  Let
  $\psi$ be as in \eqref{psi} with $h,\phi$ as in
  Lemmas~\ref{h},\ref{phi}.  Assume that the coefficients $g-\id$, $d$
  and $e$ satisfy \eqref{regg}. Then the following $L^2$ Carleman
  estimate holds for all functions $u$ supported in $\{y \leq 9\tau\}$:
  \begin{equation}
    \begin{split} 
      \delta_2^\frac12\left(\sum_{j=0,1,2}  \tau^{-\frac{j}2 }
\|a^2 e^{\psi} D^j u\| +
     \tau^{-\frac12}  \| a_{\perp}^2 e^{\psi} D_{\perp} u\|\right) 
      \lesssim \| e^{\psi} \Pt u\|.
    \end{split}
    \label{cpf}\end{equation}
  \label{ppsiflat}\end{proposition}

\begin{proof}
 After conjugation
  \[
  P_\psi := e^{\psi(s,y)} \Pt e^{-\psi(s,y)}
  \]
  we decompose $\dtpsi$ into its selfadjoint and its skewadjoint part
  \[
  P_\psi = L^r_\psi + L^i_\psi
  \]
  which for $y^2 < 9 \tau$  can be
  expressed in the generic form (see also \eqref{lopsi}):
  \[
  L^r_\psi = L^r_{0,\psi} +\partial d \partial +
   \tau d 
  \]
  \[
  L^i_\psi = L^i_{0,\psi} + \tau^\frac12 (d
  \d_j + \d_j  d )
  \]
with $d$ satisfying \eqref{regg}.
  Then \eqref{cpf} follows from
  \begin{equation}
\sum_{j=0,1,2} \tau^{-j} ( \delta_2 \|a_{int}^2 D^j v\|^2 + 
\langle  h''\rangle^\frac12 D^j v\|^2) + \delta_2 \tau^{-1}\|
      a_{\perp}^2 \nabla_{\perp} v\|^2  \lesssim \| \Ptpsi
      v\|^2.
      \label{cpfpsi}
  \end{equation}
  The proof will consist of three steps.

  {\bf Step 1:} First we show that for $v$ supported in $\{|y|^2 \leq
  9 \tau\}$ we have
  \begin{equation} \label{step1}
    \begin{split} 
      \frac{ \delta_2}{\tau} ( \|a_{int}^2 \nabla v\|^2\! +\! \| a_{\perp}^2
      \nabla_{\perp} v\|^2) \!+\! \| (h'')^\frac12 v \|^2\! +\! \| L^r_\psi
      v\|^2 \lesssim \| \Ptpsi v\|^2 \!+  \delta_1 \| a_{int}^2 v
      \|^2
    \end{split}
  \end{equation}

  We compute
  \[
  \| P_\psi v\|_{L^2}^2 = \| L^r_\psi v\|_{L^2}^2+\| L^i_\psi
  v\|_{L^2}^2 + \langle [L^r_\psi,L^i_\psi ] v, v \rangle
  \]
  We expand the commutator
  \[
  [L^r_\psi,L^i_\psi ]= [L^r_{0,\psi},L^i_{0,\psi} ] + [\partial d
  \partial + \tau d ,L^i_{0,\psi} ] + \tau^\frac12 [L^r_{\psi},d \d_j
  + \d_j d ]
  \]
  The main contribution in \eqref{step1} comes from the first
  commutator, for which we use \eqref{com0} to obtain the terms on the
  left side of \eqref{step1}.

  The second commutator is estimated by
  \[
  |\langle [M^r,L^i_{0,\psi} ] v,v\rangle| \lesssim \delta_1( \|
  a_{int}^2 v\|^2+ \tau^{-1}\|a_{int}^2 \nabla v\|^2)
  \]
  and the second term on the right is negligible since $\delta_1 \ll
  \delta_2$.  Indeed, we write
  \[
  [\partial d \partial + \tau d,L^i_{0,\psi} ] = - \partial_k q^{kl} \partial_l +
  r
  \]
  where the coefficients $q$, $r$ have the generic form
  \[
  q = d_s + \psi_y d_y + \psi_{yy} d +  d \psi_{yy}, \qquad 
    r = \partial d \partial \Delta \psi + \tau  ( d_s + \psi_y d_y)
  \]
  Using the bounds \eqref{regg} for $d$ and \eqref{phibound} for
  $\phi$ we estimate
  \[
  |q| \lesssim \delta_1 \tau^{-1} a_{int}^4, \qquad |r| \lesssim
  \delta_1 a_{int}^4.
  \]

 Finally, the third commutator is estimated in a similar fashion.
We write it in the form
\[
\tau^\frac12 [L^r_{\psi},d \d_j + \d_j d ] = - \partial_k q^{kl} \partial_l + r
\]
where the coefficients $q$, $r$ have the generic form
\[
q = \tau^\frac12(d_y + d d_y)
, \qquad r =\tau^\frac12( \Delta d_y + \partial d \partial d_y + 
d \psi_{ys} + d \psi_{yy} \psi_y  +\tau d d_y )
\]
 Using  \eqref{regg}  and \eqref{phibound} 
 we  obtain the same bounds for $q$ and $r$ as in the previous case.
This  concludes the proof of \eqref{step1}.

  {\bf Step 2:} We use an elliptic estimate to show that for functions
  $v$ supported in $\{|y|^2 \leq 9 \tau\}$ we have
  \begin{equation} \label{step2}
   \delta_2\!\left(\sum_{j=0}^2
\tau^{-j} \|a_{int}^2 D^j v\|^2 \!+\! \tau^{-1} \| a_{\perp}^2
      D_{\perp} v\|^2\!\right) \!\!+\! \| (h'')^\frac12 v \|^2\! +\! \| L^r_\psi
      v\|_{L^2}^2 \lesssim \| \Ptpsi v\|^2 
  \end{equation}

  The elliptic bound
  \[ \|D^2 v\| + \tau \|v\| \lesssim \tau^\frac12 \|D v\|+ \|(-\Delta
  -h'(s))v\|
  \]
  can easily proven by a Fourier transform. It implies
  \[
  \|D^2 v\| + \tau \|v\| \lesssim \tau^\frac12 \|D v\| + \|(H -h'
  (s))v\| + \| y^2 v\|,
  \]
  We can replace $H -h' (s)$ by $L^r_{\psi}$ due to the pointwise
  estimate
  \[
  | (L^r_{\psi}-(H -h' (s))) v| \lesssim \delta_1(|D^2 v| +
  \tau^\frac12|Dv| + \tau |v|)
  \]
  Then \eqref{step2} follows from \eqref{step1}.

  {\bf Step 3:} Here we use the spectral gap condition to improve our
  bound when $h'' \ll 1$ and show that \eqref{step2} implies \eqref{cpf}.
  It suffices to show that if $h''(s) < \frac18$ then
  \[
  \|v\|_{L^2} + \tau^{-1} \Vert D^2 v \Vert\lesssim \| L^r_\psi
  v\|_{L^2}
  \]
  Indeed, let $s \in [i,i+1]$ so that $h''(s) < \frac18$. Then
  $h^\prime$ has a positive distance from the integers.  Also $\e_i
  \lesssim 1$ which implies that at time $s$ we must have
  \[
  |g -\id| \lesssim \delta_1 \tau^{-1}, \qquad | Dg| \lesssim \delta_1
  \tau^{-1}, \qquad |\psi_r| \lesssim \delta_2 \tau^{-\frac12}.
  \]
  Hence we may think of $ L^r_\psi$ as a small perturbation of $H -
  h'(s)$ and compute
  \[
  \begin{split}
    \Vert v \Vert + \tau^{-1} \Vert D^2 v \Vert \lesssim \Vert
    (H-h'(s)) v \Vert \lesssim \Vert L_r v \Vert + (\delta_1 +
    \delta_2^2) \Vert v \Vert + \delta_1 \tau^{-1} \Vert D^2 v \Vert
  \end{split}
  \]
where the last two terms on the right are negligible compared to the 
left hand side. The proof of the proposition is concluded.

\end{proof}

We want to reformulate the previous result in a more symmetric
fashion.  To do this we weaken the estimates slightly by using a
coarser partition of the space.  We distinguish three cases for $i$
corresponding to the value of $j(i)$ in Lemma~\ref{lepsilon} (v).

\begin{definition} We define the partition $B_{ij}$ as follows. 
\label{bij} 
\begin{enumerate}
\item If $j(i) = 0$ (which corresponds to $\e_{i} \gtrsim
  \tau^{-\frac12}$)   we set
  \[
  B_{ij} = A_{ij}, \qquad b \approx a,\ b_\perp \approx a_{\perp}
  \]
\item If $0 < j(i) < [\ln \tau/2 +2]$ (which corresponds to
  $\tau^{-1} \lesssim \e_{i} \lesssim \tau^{-\frac12}$) we set
  \[
  B_{ij} = A_{ij}, \qquad b \approx a , \qquad b_\perp = a_\perp \qquad 
j \geq j(i)
  \]
  respectively
  \[
  B_{i0} = \bigcup_{j < j(i)} A_{ij}, \qquad b  \approx
  a_{|A_{ij(i)}} \qquad b_\perp =0\text{ on }
  B_{i0}
  \]
\item If $j(i)=[\ln \tau/2 +2]$ (which corresponds to
  $\tau^{-1} \lesssim \e_{i} $) we set
  \[
  B_{i0} = \bigcup_{j=0}^{[\ln \tau/2] +2}
 A_{ij}, \qquad b= 1, \quad b_\perp=0 \text{ on
  } B_{i0.}
  \]
\end{enumerate}
\end{definition} 
Heuristically the definition of the $B_{ij}$ partition is motivated 
by the fact that in regions $A_{ij}$ with $j < j(i)$ the weight $\phi$
is ineffective, i.e. it changes by at most $O(1)$. Thus the 
convexity there is useless, and instead we rely directly on 
localized bounds for the Hermite operator.

Since the $\varepsilon_{ij}$ are slowly varying $b \lesssim a $ and
$b_\perp \lesssim a_\perp$ and we may replace the $a$'s by $b$'s
in the above proposition.

To provide some bounds on the size of $b$ and $b_\perp$ we introduce a
function $1 \leq r(s) \leq \tau^\frac12$ which is smooth and slowly
varying on the unit scale in $s$ so that
\[
r(s) \approx e^{j(i)}   \qquad s \in [i,i+1]
\]
This describes the region where $b$ is tapered off and $b_\perp = 0$. 
Precisely,  consider two cases corresponding to the three cases above.

(1) If $r(s) \approx 1$ then we have the bounds
\begin{equation}
\begin{split}
M \tau (1+r)^{-\frac32} \lesssim  b^4(r,s) &\lesssim M \tau (1+r)^{-1} \\
 b_\perp^4(r,s) &\lesssim M \tau (1+r)^{-1}
\end{split}
\label{b1}\end{equation}
where the parameter $M \geq 1$ is defined by $M \approx \e(s) \tau^\frac12$.

(23) If $r(s) \gg 1$ then 
\begin{equation}
\begin{split}
 \tau r(s)^{-\frac12} (r(s)+r)^{-\frac32}  \lesssim b^4(r,s) & \lesssim 
 \tau r(s)^{-1} (r(s)+r)^{-1} \\
b^4_\perp(r,s) & \lesssim  \tau r(s)^{-1} (r(s)+r)^{-1}
\end{split}
\label{b23}\end{equation}
with approximate equality when $r \lesssim r(s)$ and 
approximate equality on the right when $r = \tau^\frac12$.

By slightly changing $b$ and $b_\perp$ we may and do assume that the
functions $b$ and $b_\perp$ are smooth with controlled derivatives.
Thus $b$ and $b_\perp$ are smooth on the unit scale in $s$ and on the
dyadic scale in $y$, and their derivatives satisfy the bounds
\begin{equation}
  |{b}_s|+ (r_s+r)|{b}_{r}|+(r_s+r)^2 |{b}_{rr}|
  \lesssim {b}, \qquad r^2 < 9 \tau
 \label{db} \end{equation}
and
  \begin{equation}
  |{b_\perp}_s|+ (r_s+r)|{b_\perp}_{r}|+(r_s+r)^2 |{b_\perp}_{rr}|
  \lesssim {b}_\perp + b  \qquad r^2 < 9 \tau
 \label{dbperp} \end{equation}
In addition we have 
\begin{equation}
\supp b_{\perp r} \subset \{r > r_s\}
\label{dbsupp}\end{equation}

Using the functions $b$ and $b_\perp$ we define the 
Banach space $X_2^0$ with norm
\[
\Vert v \Vert_{X_2^0}^2 = \Vert b v \Vert_{L^2}^2 + \tau^{-1/2} \Vert
b_\perp D_\perp^{\frac12} v \Vert_{L^2}^2
\]
Then the symmetrized version of Proposition~\ref{ppsiflat} has the
form

\begin{proposition}
  Assume that the coefficients of $P$ satisfy \eqref{regg}. Let $\psi$ be
  as in \eqref{psi} with $h,\phi$ as in Lemmas~\ref{h},\ref{phi}.
  Then the following $L^2$ Carleman estimate holds for all functions
  $u$ supported in $\{y \leq 9\tau\}$:
  \begin{equation} \label{fce} \Vert e^{\psi(s,y)} u \Vert_{X_2^0}
    \lesssim \Vert e^{\psi(s,y)} Pu \Vert_{(X_2^0)^*}
  \end{equation}
  \label{curvecor}
\end{proposition}


\begin{proof}
Conjugating with respect to the exponential weight, the bound
\eqref{fce} is rewritten in the form
\begin{equation} \label{fce1} \Vert v \Vert_{X_2^0}
    \lesssim \Vert P_\psi v \Vert_{(X_2^0)^*}
  \end{equation}
Observing that 
\[
 D_{\perp}^2 = - y^{-2} \Delta_{\mathbb{S}^{n-1}}
\]
we introduce the operator 
\[
Q = Q(|y|,(-\Delta_{\mathbb{S}^{n-1}})^{\frac12}), \qquad q(r,\lambda)
= (b^4(r) + r^{-2} \tau^{-1} b_\perp^4(r) \lambda^2)^\frac14
\]
 Then the inequality \eqref{fce1} can be written as
  \[
  \Vert Q v \Vert_{L^2} \lesssim \Vert Q^{-1} P_\psi v \Vert_{L^2}
  \]
  whereas inequality \eqref{cpf} implies
  \[
  \Vert Q^2 w \Vert_{L^2} \lesssim \Vert P_\psi w \Vert_{L^2}.
  \]

  Hence it is natural to apply \eqref{cpf} to the function $ w =
  Q^{-1} v$, which solves
  \[
  \dtpsi w = Q^{-1} \dtpsi v + Q^{-1} [Q,\dtpsi] w
  \]
 Thus \eqref{fce1} would follow provided that 
the commutator term is small,
\[
  \| Q^{-1} [Q,\dtpsi] w \|_{L^2} \ll
\|b^2 w\|_{L^2} + \tau^{-1/2} \| b^2 
  \nabla w\|_{L^2} + \tau^{-1} \| b^2 D^2_y w \|_{L^2}
\]
Unfortunately a direct computation shows that the smallness fails
when $j$ is close to $j(i)$ even  in the flat case, i.e. with $P_\psi$
replaced by 
\[
P_{0,\psi} = \d_s -\Delta +y^2 -\psi_s - \psi_y^2
\]

To remedy this we introduce an additional small parameter $\delta$ and
use it to define a modification $Q_\delta$ of $Q$. 
We modify $r(s)$ to $r_\delta(s)$ defined by
\[
r_\delta(s)^{-2} = \delta^8 r(s)^{-2} + \delta^2 \tau^{-1} 
\]
and use it to define the function 
\[
b_\delta(r,s)^4 =  \delta^{-12} \tau (r^2 +r_\delta(s)^2)^{-1} 
\]
We can still compare it with $b$,
\[
b_\delta(r,s)^4 \lesssim \delta^{-4} b(r,s)^4
\]
Its usefulness lies in the fact that it is larger than $b$ exactly in
the region where the commutator term above is not small.

The modification $Q_\delta$ of $Q$ has symbol
\[
q_\delta(r,s,\lambda) = q(r,s,\lambda) + b_\delta(r,s)= 
(b^4(r,s) + r^{-2} \tau^{-1}
b_\perp^4(r,s) \lambda^2)^\frac14 + b_\delta(r,s)
\]
which satisfies
\begin{equation}
q \leq q_\delta \lesssim \delta^{-1} q
\label{qdeltabd}\end{equation}
We claim that it satisfies the bound
\begin{equation}
  \| Q_\delta^{-1} [Q_\delta,\dtpsi] w \|_{L^2} \lesssim 
(\delta + c(\delta) \delta_1)\sum_{j=0,1,2}
\tau^{-\frac{j}2} \|b^2 D^j w\|_{L^2} 
\label{qqcom}\end{equation}
Suppose this is true. Then we fix $\delta$ sufficiently small, and for
$\delta_1$ small enough we apply \eqref{cpf} to $w = Q_\delta^{-1} v$.
By \eqref{qqcom} he commutator term in the equation for $w$ can be
neglected, and we obtain
\[
\| Q^2 w\|_{L^2} \lesssim \|Q_\delta^{-1} P_\psi v\|_{L^2}
\]
which by \eqref{qdeltabd} implies that
\[
 \Vert Q v \Vert_{L^2} \lesssim \delta^{-1} \Vert Q^{-1} P_\psi v \Vert_{L^2}
\]
It remains to prove \eqref{qqcom}.  

I. We first calculate the commutator in the flat case, i.e. with
$P_\psi$ replaced by $P_{0,\psi}$.  Due to the spherical symmetry the
only contribution comes from the radial part of the Laplacian and the
$s$ derivative. Hence using polar coordinates we compute
 \[
  Q_\delta^{-1} [Q_\delta,P_{0,\psi}] =  Q_\delta^{-1} \Big( Q_{\delta
    rr} +
  \frac{n-1}r Q_{\delta r} + 2 Q_{\delta r} \d_r  - Q_{\delta s} \Big)
  \]
  Then is suffices to verify that on the symbol level we have
\begin{equation}
  |q_{\delta rr}|+ r^{-1} |q_{\delta r}|+|q_{\delta s}| +\tau^{1/2}
  |q_{\delta r}| \lesssim \delta q_\delta b^2
  (1+\tau^{-1} r^{-2} \lambda^2)^\frac12
\label{qrrqr}\end{equation}

I.(1). We begin with the $q$ component of $q_\delta$. Using \eqref{db} and 
\eqref{dbperp} one obtains
\[
 |q_{ rr}|+ r^{-1} |q_{ r}|+|q_{s}| +\tau^{1/2}
  |q_{ r}| \lesssim  (r(s)+r)^{-1} \tau^{\frac12} q
\]
Thus it remains to show that
\[
(r(s)+r)^{-1} \tau^{\frac12} q \lesssim \delta q_\delta b^2
  (1+\tau^{-1} r^{-2} \lambda^2)^\frac12
\]
Optimizing with respect to $\lambda$ it suffices to consider 
the cases $\lambda = 0$ respectively  $\lambda = r \tau^\frac12$, where the
above inequality becomes
\[
(r(s)+r)^{-1} \tau^{\frac12} (b+b_\perp)  \lesssim \delta (b_\delta +
b) b^2
\]
or equivalently
\[
 (b+b_\perp)  b_1^2 \lesssim \delta (b_\delta + b) b^2
\]
which is true since by \eqref{b1} and \eqref{b23} we have $b_\perp b_1
\lesssim b^1$ while  $b_1 \lesssim \delta b_\delta$.

I.(2). Next we consider the $b_\delta$ component of $q_\delta$, for which it suffices 
to prove that
\begin{equation}
|b_{\delta rr}|+ r^{-1} |b_{\delta r}|+|b_{\delta s}| +\tau^{1/2}
  |b_{\delta r}| \lesssim \delta b_\delta b^2
\label{bdelta}\end{equation}

I.(2).(a). For the $s$ derivative we compute 
\[
\frac{b_{\delta s}}{b_\delta} = \frac{(r_\delta^2(s))_s}{r^2 +r_\delta^2(s)}
\]
therefore we want to show that
\[
(r_\delta^2(s))_s \lesssim \delta (r^2 +r_\delta^2(s)) b^2
\]
We optimize the right hand side with respect to $r$. The minimum is attained
when  $r^2 = \min\{  r_\delta^2(s),\tau\}$.
 We need to consider
two cases:

I.(2).(a).(i). If $ r_\delta(s) \lesssim \tau^{\frac12}$ then
$r_\delta^2(s) \approx \delta^{-8} r^2(s)$ and $\tau > \delta^{-8}
r^2(s)$. Hence using the estimate from below in \eqref{b1} and
\eqref{b23} we obtain $b^4(r_\delta) \gtrsim \delta^{-2}$.
 Then the above bound for $r = r_\delta$  follows since $|(r_\delta(s)^{-2})_s| \lesssim
r_\delta(s)^{-2}$.

I.(2).(a).(ii). If $r _\delta(s) \gtrsim \tau^\frac12$ then
by\footnote{the equality holds on the right when $r^2=\tau$}
\eqref{b23} we evaluate $b^2(\tau^\frac12) \approx \tau^{\frac14}
r(s)^{-\frac12}$.  Then the above bound becomes
\[
\delta^8 (r(s)^{-2})_s  \lesssim \delta r_\delta^{-2}   \tau^{\frac14} r(s)^{-\frac12}
\]
Since $ |(r(s)^{-2})_s| \lesssim r(s)^{-2}$ it suffices to show that
\[
\delta^8  r(s)^{-2}  \lesssim \delta r_\delta^{-2}   \tau^{\frac14} r(s)^{-\frac12}
\]
The worst case is $r(s)^2 = \delta^6 \tau$, $r_\delta(s)= \delta^{-2}
\tau$ when it is verified directly.

I.(2).(b). For the $r$ derivatives the last term is the worst. Since
\[
\frac{b_{\delta r}}{b_\delta} = \frac{2r}{r^2 +r_\delta^2(s)}
\]
we want to show that
\[
\tau^{\frac12} r  \lesssim \delta (r^2 +r_\delta^2(s)) b^2
\]
Optimizing with respect to $r$ the worst case is when $r^2 = \min\{  r_\delta^2(s),\tau\}$.

I.(2).(b).(i). If $ r_\delta^2(s) \lesssim \tau$ then $r_\delta(s)\approx
\delta^{-4} r(s)$ therefore for $r = r_\delta(s)$ the above
relation becomes
\[
\tau^{\frac12}  \lesssim \delta^{-3} r(s) b^2(r_\delta(s))
\]
which follows from the bound from below in \eqref{b1} and \eqref{b23}.

I.(2).(b).(ii). If $ r_\delta^2(s) \gtrsim \tau$ then as before we
evaluate $b^2(\tau^\frac12) \approx \tau^{\frac14} r(s)^{-\frac12}$
and rewrite the above bound as
\[
\tau  \lesssim \delta r_\delta^2(s)   \tau^{\frac14} r(s)^{-\frac12}
\]
The right hand side is smallest either when $r_\delta(s)=\tau^\frac12$ 
and $r(s) = \delta^4 \tau^\frac12$ or when $r_\delta(s)=\delta^{-1} \tau^\frac12$ 
and $r(s) =  \tau^\frac12$. In both cases the inequality is easily verified.

II.  Now we deal with the general case, which we treat as a
perturbation.  Since we do not care about the dependence of the
constants on $\delta$ to keep the notations simple we include
$b_\delta$ in $b$ and work with $Q$ instead of $Q_\delta$.  Thus in
the computations below we allow the implicit constants to depend on
$\delta$.

 Suppose that $A$ is a pseudodifferential operator of order $1$ and
  let $\eta$ be any Lipschitz function. Then
  \begin{equation}
 \| [A,\eta] f \|_{L^2} \lesssim \Vert f
    \Vert_{L^2}.
  \end{equation}  
  We write
  \[
 {q}(\lambda ) = b + (b^4 + r^{-2} b_\perp^4 \lambda^2 /\tau
  )^{\frac14} -b =: b + {q_1}(\lambda).  
\] 
Even though $q_1$ has order $\frac12$, we treat it as an
operator of order $1$ and estimate
\[
\sup_\lambda \langle \lambda\rangle^{k-1} |{q_1}^{(k)}(\lambda)|
\lesssim \frac{b_\perp^2}{r b\tau^{\frac12}}
\] 
Hence, for each $r$ we obtain the bound on the sphere  $S^{n-1}$
  \begin{equation}
    \Vert [Q,\eta] f \Vert_{L^2} \lesssim  \frac{b_\perp^2}{r  b\tau^{\frac12}}
\Vert \eta \Vert_{Lip(S^{n-1})}  \Vert f \Vert_{L^2}. 
\label{comS0}  \end{equation}
As a consequence, it also follows that
 \begin{equation}
    \Vert [Q,\eta \nabla_\theta^j] f \Vert_{L^2} \lesssim  
\frac{b_\perp^2}{r  b\tau^{\frac12}}
\Vert \eta \Vert_{Lip(S^{n-1})}  \Vert \nabla_\theta^j f \Vert_{L^2}. 
\label{com12}  \end{equation}
where $\nabla_\theta$ stands for the vector fields $x_i
\partial_j-x_j \partial_i$ generating the tangent space of $S^{n-1}$.

To use these bounds we write the difference $P_\psi -P^0_\psi$
in polar coordinates,
\[
P_\psi -P^0_\psi = P_\theta^0 \partial_r^2 + P_\theta^1 \partial_r 
+ P_\theta^2
\]
where $P_\theta^j$ are spherical differential operators of order $j$.
Modulo zero homogeneous coefficients which are polynomials in $x
r^{-1}$ we can write
\[
P_\theta^0 = d, \qquad P_\theta^1 = d r^{-1} \nabla_\theta+
\tau^\frac12 d + d_y 
\]
\[
 P_\theta^2 = d r^{-2} \nabla_\theta^2 +
( \tau^\frac12 d +d_y ) r^{-1} \nabla_\theta+
\tau d + \tau^\frac12 d_y
\]
where $d$ stands for coefficients satisfying \eqref{regg}.
In the support of $b_\perp$ we have $a \approx b$ therefore
our regularity assumptions on $d$ show that for fixed $r$ we have
\[
\| d \|_{L^\infty} + \| d \|_{Lip(S^{n-1})} \lesssim \delta b^4 \tau^{-1}
\]
The coefficients involving $d_y$ satisfy better Lipschitz bounds and are
neglected in the sequel.

We expand the commutator
\[
[Q,\dtpsi] = \sum_{j=0,1,2} [Q,P_\theta^j] \partial_r^{2-j} +
P_\theta^0 (Q_{rr} + 2Q_r \partial_r) + P_\theta^1 Q_r
\]

Using the  trivial $b^{-1}$ bound for $Q^{-1}$ and \eqref{com0},
\eqref{com12} we estimate the first term,
\begin{eqnarray*}
\sum_{j=0,1,2} \| Q^{-1} [Q,P_\theta^j] \partial_r^{2-j} w\|_{L^2}
\lesssim b^{-1}  \delta b^4 \tau^{-1}  \frac{b_\perp^2}{r
  b\tau^{\frac12}}\sum_{j=0,1,2} 
 \tau^{1-\frac{j}2} \| D^j w\|_{L^2} 
\end{eqnarray*}
This is bounded by the right hand side in \eqref{qqcom}
since
\[
b_\perp^2 \lesssim r \tau^\frac12
\]
The second term in the commutator is estimated by
\[
\| Q^{-1}P_\theta^0 (Q_{rr} +
2Q_r \partial_r) w\|_{L^2} \lesssim \delta 
( \|Q_{rr} w\|_{L^2}  + \|Q_r \partial_r w\|_{L^2})
\]
This is bounded by the right hand side in \eqref{qqcom}
provided that
\[
|q_{rr}| + \tau^{-\frac12} |q_r| 
\lesssim b^2(1+\tau^{-\frac12} r^{-1}
\lambda)
\]
which follows from \eqref{qrrqr}. The third term in the commutator 
is treated similarly. This concludes the proof of the proposition.
\end{proof}

To conclude our study of the $L^2$ Carleman estimates 
we need to also pay some attention to elliptic estimates.
The conjugated operator $P_\psi$ is elliptic in the region 
$\{ y^2 +\xi^2 \geq 4 \tau \}$. Precisely, in this region we 
have the symbol bound
\[
|L_\psi^r (s,y,\xi)| \gtrsim y^2 +\xi^2
\]
Consequently, we can improve our estimates in this region.
We consider a smooth symbol $a_e(y,\xi)$  with the following
properties
\[
\begin{split}
\supp a_e &\subset \{ y^2 +\xi^2 \geq 8 \tau\}
\\
a_e(y,\xi) & = (y^2 +\xi^2)^\frac12 \qquad  \text{ in }\ \{y^2 +\xi^2 \geq 9 \tau\}
\end{split}
\]
We define the space $X_2$ with norm
\begin{equation} \label{X2} 
\| v\|_{X_2}^2 = \| v\|_{X_2^0}^2 + \| a_e^w(y,D) v \|^2
\end{equation}
The dual space $X_2^*$ has norm
\begin{equation}\label{X2*}
\|f\|_{X_2^*}^2 = \inf \{ \|f_1\|_{(X_2^0)^*}^2 +\|f_2\|^2; \ f = f_1+
a_e^w(y,D) f_2 \}
\end{equation} 
We note that due to the elliptic bound for high frequencies,
we also have the dual bounds
\begin{equation}
\tau^{-\frac12} \| b D v\| \lesssim \|v\|_{X^2}, \qquad 
\|\nabla f\|_{X_2^*} \lesssim \|b^{-1} f\|
\label{bd}\end{equation}

Then our final $L^2$ Carleman estimate is

\begin{theorem}
 Assume that the coefficients of $P$ satisfy \eqref{regg}. Let $\psi$ be
  as in \eqref{psi} with $h,\phi$ as in Lemmas~\ref{h},\ref{phi}.
  Then the following $L^2$ Carleman estimate holds for all functions
  $u$ for which the right hand side is finite:
  \begin{equation} \label{fcemain} 
\Vert e^{\psi(s,y)} u \Vert_{X_2}
    \lesssim \Vert e^{\psi(s,y)} Pu \Vert_{X_2^*}
  \end{equation}
\label{core}\end{theorem}

\begin{proof}

  We first prove the result using the stronger assumption \eqref{regg}
  on the coefficients. After conjugation  we have to show that 
 \begin{equation} \label{fcemainv} 
\Vert v \Vert_{X_2^0}
    \lesssim \Vert  P_\psi v \Vert_{X_2^*}
  \end{equation}

  We consider two overlapping smooth cutoff symbols $\chi_i=
  \chi_i(y^2+\xi^2)$ and $\chi_e= \chi_\e(y^2+\xi^2)$.  The interior
  one $\chi_i$ is supported in $ \{ y^2 +\xi^2 \leq 7 \tau\}$ and
  equals $1$ in $\{ y^2 +\xi^2 \geq 6 \tau\}$.  The exterior one
  $\chi_e$ is supported in $ \{ y^2 +\xi^2 \geq 4 \tau\}$ and equals
  $1$ in $\{ y^2 +\xi^2 \leq 5 \tau\}$.  We need the following bounds
  for $\chi_i$ and $\chi_e$:

\begin{lemma}
a) The operator $\chi_i(x,D)$ satisfies the bound
\begin{equation}
\|\chi_i(x,D) f \|_{(X_2^0)^*}  \lesssim \| f\|_{X_2^*}
\label{x2x20}\end{equation}

b) The operators $\chi_i(x,D)$ and $\chi_e^w(x,D)$ satisfy the 
following commutator estimates: 
\begin{equation}
\| b^{-1} [\chi_i(x,D),P_\psi] v\|  \lesssim \tau^{-\frac14} \| b v\|
+ \| \chi_e v\|
\label{chiicom}\end{equation}

\begin{equation}
\|  [\chi_e^w,P_\psi] v\|  \lesssim \tau^{\frac18} \| b v\|
\label{chiecom}\end{equation}
\end{lemma}

\begin{proof}
a) By duality the bound \eqref{x2x20} is equivalent to
\[
\|\chi_i(x,D) v \|_{X_2} \lesssim \|v\|_{X_2^0}
\]
We have 
\[
\|a_e(x,D) \chi_i(x,D) v \| \lesssim \tau^{-N} \|f\|
\]
since  the supports of the symbols $(1-\chi_e(x,\xi))$ and
$a_e(x,\xi)$ are $O(\tau^\frac12)$ separated. Then it remains to show
that
\[
\|\chi_i(x,D) v \|_{X_2^0} \lesssim \|v\|_{X_2^0}
\]
which is fairly straightforward and is left for the reader.

b) We now consider the bound \eqref{chiicom}. Commute first
$\chi_i$ with $\partial_s +H- h'(s)$. We have
\[
 [\chi_i(x,D), \partial_s +H- h'(s)]= [\chi_i(x,D),H]
\]
Since $\chi_e = 1$ in the support of $\nabla_{x,\xi} \chi_i$ and the
Poisson bracket of $\chi_i$ and $x^2+\xi^2$ vanishes, by standard pdo
calculus we obtain
\[
\|  [\chi_i(x,D), \partial_s +H- h'(s)] v\| \lesssim \|\chi_e v\| +
\tau^{-N} \| v\|
\]

The difference $P_\psi - (\partial_s +H- h'(s))$ can be 
expressed in the form 
\[
P_\psi - (\partial_s +H- h'(s)) = \partial g \partial +
\tau^{\frac12}( g \partial +\partial g) + \tau g
\]
where the function $g$ satisfies the bounds
\[
|g| + \langle y\rangle |g_y| + \langle y\rangle |g_{yy}| \lesssim \e_i 
\]
These lead to an estimate for fixed $s \in [i,i+1]$, 
\[
 \| b^{-1} [\chi_i(x,D),P_\psi - (\partial_s +H- h'(s))] v \|\lesssim \e_i
 \tau^\frac12 \| \langle y\rangle^{-1} b^{-1} v \|
\]
Then \eqref{chiicom} follows since
\[
\e_i \tau^\frac12  \langle y\rangle^{-1} \lesssim \tau^{-\frac14} b^2
\]

Finally, the proof of the estimate \eqref{chiecom} is similar but
simpler.

\end{proof}

We continue with the proof of the proposition.  For the nonelliptic
part we apply \eqref{fce} to the function $\chi_i(x,D)v$ which is
supported in $\{y^2 < 9\tau\}$. This gives
\[
\|\chi_i(x,D)v\|_{X_2^0} \lesssim \| \chi_i(x,D)P_\psi
v\|_{(X_2^0)^*} + \| [  \chi_i(x,D),P_\psi] v\|_{L^2}
\]
For the first term on the right we use the bound \eqref{x2x20} while
for the second we use \eqref{chiicom}.  This yields
 \begin{equation}
\| \chi_i(x,D))v\|_{X_2^0} \lesssim \| P_\psi
v\|_{X_2^*} + \tau^{-\frac14}   \|bv\| +\|\chi_e v\|
\label{none}\end{equation}

On the other hand for the estimate in the elliptic region
we compute
\begin{equation}
\langle (\chi_e^w)^2 v, P_\psi v \rangle = \langle \chi_e^w v,
L_\psi^r \chi_e^w v\rangle + \langle \chi_e^w v, [\chi_e^w,P_\psi^r] v
\rangle 
\label{ellreg}\end{equation}
For the first term we split $L_\psi^r$ into $H-h'$ plus a
perturbation. Using pointwise bounds for the coefficients of $P_\phi$
we obtain
\[
| L_\psi^r v - (H-h') v| \lesssim \delta_1( (\tau+y^2) |v| +
\tau^\frac12 |Dv| + |D^2 v|)
\]
which shows that
\[
 \langle \chi_e^w v, L_\psi^r \chi_e^w v\rangle =  \langle \chi_e^w v,
 (H-h') \chi_e^w v\rangle + O(  \delta_1 \langle \chi_e^w v, (H+\tau) \chi_e^w v\rangle)
\]
The symbol of $H-h'$ is elliptic in the support of $\chi_e$,
therefore a standard elliptic argument yields
\[
 \langle \chi_e^w v, (H+\tau) \chi_e^w v\rangle \lesssim  \langle \chi_e^w v,
 (H-h') \chi_e^w v\rangle + C \tau^{-N} \|v\|^2
\]
for a large constant $C$. This further gives 
\[
\langle \chi_e^w v, (H+\tau) \chi_e^w v\rangle \lesssim  \langle \chi_e^w v, L_\psi^r \chi_e^w v\rangle
+ C \tau^{-N} \|v\|^2
\]

Returning to \eqref{ellreg}, we obtain
\[
c \langle \chi_e^w v, (H+\tau) \chi_e^w v\rangle \leq  - \langle
(\chi_e^w)^2 v, P_\psi v \rangle  + \langle \chi_e^w v, [\chi_e^w,P_\psi^r] v
\rangle + C \tau^{-N} \|v\|^2
\]
We use \eqref{chiecom} and then the Cauchy-Schwartz inequality to
obtain
\[
\langle \chi_e^w v, (H+\tau) \chi_e^w v\rangle \lesssim \|
(H+\tau)^{-\frac12} P_\psi v\|^2 + \tau^{-\frac34} \|bv\|^2  
\]
The first term on the right is properly controlled  due to the straightforward estimate
\[
\| (H+\tau)^{-\frac12}  f\| \lesssim \|f\|_{X^*_2}
\]
Hence combining the above inequality with \eqref{none} we obtain
\[
\| \chi_i^(x,D))v\|_{X_2^0} + \| (H+\tau)^\frac12 \chi_e^w(x,D) v\|
\lesssim \|P_\psi v\|_{X_2^*} + \tau^{-\frac34} \|bv\| +
\|\chi_e^w(x,D) v\|
\]
The last two terms on the right are negligible compared to the left hand
side, therefore we obtain
\begin{equation}
\|v\|_{X_2^0} + \|  (H+\tau)^\frac12 \chi_e^w(x,D) v\|
\lesssim \|P_\psi v\|_{X_2^*}
\end{equation}
Then \eqref{fcemainv} follows since $\chi_e=1$ in the support of $a_e$.

It remains to show that the assumption \eqref{regg} on the
coefficients for \eqref{curvecor} can be replaced by the weaker
condition \eqref{g}. This is a direct consequence of \eqref{bd}
combined with the following regularization result:

\begin{lemma}
Let  $d$ be a function which satisfies  \eqref{g}.
Then there is an approximation  $g_1$ of it satisfying \eqref{regg}
so that
\[
| g - g_1| \lesssim b^2 \tau^{-1}
\]
\end{lemma}

\begin{proof}
First we transfer \eqref{g} to the $(s,y)$ coordinates.
A short computation yields the equivalent form
\begin{equation}
\|d\|_{L^\infty(A_{ij})} + e^j \|d\|_{Lip_y(A_{ij})} +
\|g\|_{C_t^{m_{ij}}(A_{ij})} \lesssim \e_{ij}
\label{gg1}\end{equation}
where the new continuity modulus $\tilde m_{ij}$ is given by
\[
\tilde m_{ij}(\rho) = \rho + e^{-\frac{2j}3} \rho^\frac13
\]

Within $A_{ij}$ we regularize $d$ in $y$  on the $\delta y = \tau^{-\frac12}$
scale and in $s$ on the $\delta s = e^{\frac{j}2} \tau^{-\frac34}$ scale,
\[
d_1 = S_{ < \tau^{\frac12}}(D_y) S_{ < e^{\frac{j}2}
  \tau^{-\frac34}}(D_s) d
\]
These localized regularizations are assembled together 
using a partition of unit corresponding to $A^{ij}$.
In $A_{ij}$ we compute
\[
|d  - d_1| \lesssim \e_{ij}( e^{-j} \delta y + m_{ij}(\delta s))
\approx \e_{ij}( e^{-j} \tau^{-\frac12} + e^{-\frac{j}{2}} \tau^{-\frac14})
\lesssim \e_{ij}^\frac12 e^{-\frac{j}2} \tau^{-\frac14} \lesssim b^2 \tau^{-1}
\]
while
\[
|\partial_s g_1| \lesssim \e_{ij} \frac{m_{ij}(\delta s)}{\delta s}
\approx \e_{ij} e^{-j} \tau^{\frac12}  
\]
The bounds for higher order derivatives of $g_1$ follow trivially due 
to the frequency localization.
\end{proof}
\end{proof}

\section{$L^p$ Carleman estimates for variable coefficient 
operators}

The variable coefficient counterpart of Proposition~\ref{plpflat} uses
the more convex weights constructed in Section~\ref{convex}. For
convenience we write it in the $(s,y)$ coordinates. Let $\tau >> 1$ and  $\mathcal{B}(\tau)$ be as in \eqref{parB1},\eqref{log} and \eqref{parB2}. 

We define the function space $X$ through its norm 
\begin{equation} 
\Vert v \Vert_{X}:= 
\|v\|_{X_2}+ \Vert v \Vert_{l^2(\mathcal{B}(\tau);L^\infty_t L^2_x)}
+ \Vert v \Vert_{l^2(\mathcal{B}(\tau);L^p_t L^q_x)}
\end{equation} 
where $(p,q)$ is an arbitrary  Strichartz pair, with $X_2$ as defined 
in \eqref{X2}.

Its (pre)dual space has the norm 
\begin{equation} 
\begin{split} 
\Vert f \Vert_{X^*}=  
\inf_{f= f_1+f_2+f_3} \|f_1\|_{X_2^*} + \|f_2\|_{L^1L^2} +
\|f_3\|_{L^{p'} L^{q'}}
\end{split} 
\end{equation} 
Then we have the following improvement of Theorem \ref{core}.

\begin{theorem}
There exists $\psi$ as in \eqref{psi} with $h$ and $\phi$ as in Lemma \ref{h} 
and \ref{phi}. 
   Then the  following  estimate
  holds for all compactly supported sufficiently regular functions $u$. 
\begin{equation}
\Vert e^\psi u \Vert_X \lesssim \Vert e^\psi \tilde P u \Vert_{X^*} 
\end{equation}
\label{mainys}\end{theorem}

The relation between $\psi$ and the partition $\mathcal{B}(\tau)$
remains a bit mysterious at this level. If we replace it by the empty
partition then the statement remains true for all $\psi$ with $h$ and
$\phi$ as in Lemma \ref{h}. The same is true for a partition into time
slices of size $1$. The convexity properties of $\phi$ allow a
localization to the finer partition $B_{ij}$ as in \ref{bij} (and, as
we shall soon see, to an even finer partition). q It is possible to
choose $\phi$ and $h$ so that the partition $(B_{ij})$ is finer than
the one defined by $\mathcal{B}(\tau)$. We assume in the sequel that 
$\psi$ has been chosen with these properties.

\begin{proof} 
As usual  this is equivalent to proving a bound from below 
for the conjugated operator,
\begin{equation}
\Vert v \Vert_X \lesssim \Vert  P_\psi v \Vert_{X^*} 
\end{equation} 

The main step in the proof is to produce a parametrix for $P_\psi$.
The key point is that the parametrix is allowed to have a fairly large
$L^2$ error. This is because $L^2$ errors can be handled by
Theorem~\ref{core}.  The advantage in having a large $L^2$ error is
that it permits to localize the parametrix construction to relatively
small sets, on which we can freeze the coefficients and eventually
reduce the problem to the case of the Hermite operator. The properties
of the parametrix are summarized in the following

\begin{proposition}
  a) Under the assumptions of the theorem there exists a parametrix
  $T$ for $P_\psi$ with the following properties:
\begin{equation} 
\Vert Tf \Vert_{X} \lesssim \Vert f \Vert_{X^*} 
\label{Tb}\end{equation}
and
\begin{equation} \label{TP} \Vert P_\psi T f -f \Vert_{X_2^*} \lesssim
  \Vert f \Vert_{X^*}
\end{equation} 

b) The same result holds with $P_\psi$ replaced by $P_\psi^*$.
\label{param}\end{proposition}

We first use the proposition to conclude the proof of the Theorem.
Let
\[
P_\psi v = f+g, \qquad \|f\|_{X_2^*} + \|g\|_{L^1L^2+L^{p'} L^{q'}}
\approx \|P_\psi v\|_{X^*}
\]
With $T$ as in part (a) of the proposition we set
\[
w = v - T g, \qquad P_\psi w = f + g - P_\psi T g 
\]
By \eqref{Tb} we can bound $Tg$ in $X$, therefore it suffices
to bound $w$ in $X$. On the other hand by \eqref{TP}
we obtain
\[
\|  P_\psi w\|_{X_2^*} \lesssim \|f\|_{X_2^*} + \|g - P_\psi T g \|_{X_2^*}
\lesssim \|  P_\psi v\|_{X^*}
\]
It remains to show that 
\[
\|w\|_X \lesssim \| P_\psi w\|_{X_2^*}
\]
By Theorem~\ref{core} we can estimate the $X_2$ norm of $w$ and
replace this with the weaker bound
\begin{equation}
  \|w\|_{L^\infty L^2 \cap L^p L^q} \lesssim 
\|w\|_{X_2}+ \| P_\psi w\|_{X_2^*}
\label{wx2}\end{equation}
This is proved using a duality argument and the parametrix $T$ for
$P_\psi^*$ given by part (b) of the proposition. For $f \in
X^*$ we write
\[
\begin{split} 
\langle w,f\rangle =& \langle w,  P_\psi^* T f\rangle + \langle w,
f-P_\psi^* T f\rangle
\\
= &  \langle P_\psi w,  T f\rangle + \langle w,
f-P_\psi^* T f\rangle
\end{split}
\]
Using both \eqref{Tb} and \eqref{TP} with $P_\psi$ replaced by 
$P_\psi^*$ we obtain
\[
| \langle w,f\rangle| \lesssim (\|  P_\psi w\|_{X_2^*} + \|w\|_{X^2})
\| f \Vert_{X^*} 
\]
and \eqref{wx2} follows. This concludes the proof of  Theorem \ref{mainys}.
\end{proof} 

It remains to prove the Proposition \ref{param}.

\begin{proof}[Proof of Proposition~\ref{param}.]
The strategy for the proof is simple: On sufficiently small sets we can
approximate the problem by one with constant coefficients and the 
properties of the parametrix follow from Section \ref{resolvent}. 
We use a partition of unity to construct a global parametrix from 
local ones. We obtain $L^2$ errors from 
\begin{enumerate} 
\item Commuting cutoff functions with the operator. Hence the partition has to
  be sufficiently coarse.  
\item Approximating the variable coefficient operator by constant coefficient
  operators.
Hence the partition has to be sufficiently fine.  
\end{enumerate} 

To elaborate on this we define the notion of a local parametrix: 

\begin{definition}[Local parametrix] 
Given a convex set $B$ we call $T$ a ($B$-) local parametrix for $P_\psi$
if for all $f$ supported in $B$ 
\begin{equation} 
\Vert Tf \Vert_{X} \lesssim \Vert f \Vert_{X^*}, 
\label{Tb1}\end{equation}
\begin{equation} \label{TP1} \Vert P_\psi T f -f \Vert_{X_2^*} \lesssim
  \Vert f \Vert_{X^*}
\end{equation} 
and $Tf$ is supported in $2B$.
\end{definition} 

If $T$ is a parametrix and $\eta$ is supported on $2B$, $\eta=1$ on $B$ 
then $\psi T$ is a local parametrix, but with constants depending on the 
commutator of $P_\psi$ and $\eta$. Vice verse, 
if  $(B_j)$ is a covering, $(\eta_j)$ a subordinate partition of $1$ 
and $T_j$ are local parametrices  then 
\[ T = \sum_j T_j \eta_j \]
is a global parametrix, because \eqref{Tb1} is obtained by summation, and 
\[ P_\psi \sum_j T_j (\eta_j f)- f = \sum_j (P_\psi T_j \eta_j f-\eta_j f \]
 provided 
\[ \sum_j \Vert \eta_j f \Vert^2_{X^*} \lesssim \Vert f \Vert^2_{X^*} \]
and its adjoint 
\[ \Vert u \Vert_{X}^2 \lesssim \sum_j \Vert \eta_j u \Vert_{X}^2. \]
This is obvious for the $L^2$ part and has to be checked for the other
part.  This strategy of constructing local parametrices leads, if it
is possible, to estimates which are stronger than in Proposition
\ref{param} and Theorem \ref{mainys}, because we may replace the
function space $X$ by $l^2 X(B_j)$ respectively. $l^2 X^*(B_j)$.

 In the first part of the proof we study the localization, and in the
 second part we provide the local parametrices.
 
\subsection{Localization scales.} Here we introduce a localization scale
which is finer than the $B_{ij}$ partition of the space, and show that
it suffices to construct the parametrix in each of these smaller sets.
Precisely, the sets $B_{ij}^k$ introduced below are the smallest sets
to which one can localize the $L^2$ estimates for the operator
$P_\psi$. The choice of their size is not yet apparent at this point,
but will become clear in the very last step of the proof, where we
estimate the commutator of $P_\psi$ with cutoff functions on such
sets. We consider three cases depending on the size of $\e_i$.

\begin{enumerate} 
\item  If $\e_i \leq \tau^{-1}$ then we use $B_{i0}$ as it is.
\item  If $\tau^{-1} \leq \e_i \leq \tau^{-\frac12}$ then we partition
the set $B_{i0}$ into time slices $B_{i0}^k$ of thickness 
\[
\delta s =
b_{i0}^{-2}.
\]
\item  If $\tau^{-1} \leq \e_i$ and $j \neq 0$ then we partition $B_{ij}$
into subsets $B_{ij}^k$ which have the time scale, radial scale and 
angular scale given by 
\[
\delta s = b_{ij}^{-2}, \qquad \delta y = \tau^{\frac12} b_{ij}^{-2},
\qquad \delta y^\perp  = \tau^{\frac12}   b_{ij,\perp}^{-2}
\]
\end{enumerate}

This gives a decomposition of the space 
\[
\R \times \R^n = \bigcup B_{ij}^k
\]
We also consider a subordinated partition of unity
\[
1 = \sum \chi_{ij}^k
\]
Suppose that in each set $B_{ij}^k$ we have a parametrix 
$T_{ij}^k$ satisfying \eqref{Tb} and \eqref{TP}. Then we define
the global parametrix $T$ by
\[
T = \sum T_{ij}^k  \chi_{ij}^k
\]

We have by an iterated application of Minkowski's inequality  
\[
\| \chi_{ij}^k f\|_{l^2 (L^1L^2+L^{p'} L^{q'})} \lesssim
\|f\|_{L^1L^2+L^{p'} L^{q'}}
\]
and the dual bound
\[
\| \sum T_{ij}^k  \chi_{ij}^k f\|_{L^\infty L^2 \cap L^p L^q}
\lesssim \|T_{ij}^k  \chi_{ij}^k f\|_{l^2(L^\infty L^2 \cap L^p L^q)}.
\]
Hence  \eqref{Tb} for $T$ would follow if we proved that
\[
\| \sum v_{ij}^k \|_{X_2} \lesssim \| v_{ij}^k\|_{l^2 X_2}
\]
where $v_{ij}^k =  T_{ij}^k  \chi_{ij}^k f$ are supported in
$B_{ij}^k$. This is trivial by orthogonality for the $L^2$ component
of the $X_2$ norm. It is also straightforward for the elliptic part 
since the kernel of the operator in the following sense:
Let $\chi_0 \in C^\infty(\R)$ be supported in $[-3/2,3/2]$, identically $ 
1$ in $[-1,1]$. We define $\chi_e (x,\xi) = \chi_0( 5- (x^2+\xi^2)/\tau^2) $.  
Then 
\[
\| \sum \chi_e^w v_{ij}^k \|_{X_2} \lesssim \| \chi_e^w v_{ij}^k\|_{l^2 X_2}
\]
and the adjoint estimate holds since the kernel of 
 $\chi_e^w$  is rapidly decreasing
beyond the $\tau^{-\frac12}$ scale, which is much smaller than the 
smallest possible spatial size for $B_{ij}^k$, namely $\delta y
\gtrsim \tau^{-\frac14}$.

It remains to consider the angular part of the $X_2$ norm, which is
best described using the spherical multiplier $Q$ appearing in
Proposition~\ref{curvecor}.  The symbol of $Q$ is smooth with respect
to $\lambda$ on the $\tau^\frac12 r b^2 b_\perp^{-2}$ scale therefore
its kernel is rapidly decaying on the angular scale $\delta \theta =
\tau^{-\frac12} r^{-1} b^{-2} b_\perp^2$ which corresponds to $\delta
y_\perp = \tau^{-\frac12} b_\perp^2 b^{-2}$.  But by \eqref{e3} and
\eqref{ij0} this can be no larger than $\tau^{-\frac38}$ which is
again much smaller than the smallest possible spatial size for
$B_{ij}^k$. Thus orthogonality arguments still apply.

Then we have
\[
[\dt,  \chi^{ij}] u =  2\chi^{ij}_y    u_y + \chi^{ij}_{yy} u +
\chi^{ij}_s   u
\]
We claim that the right hand side is negligible in the estimate.
For this it suffices to verify that 
\[
|\chi^{ij}_y| \ll \tau^{-\frac12} b_{ij}^2, \qquad
|\chi^{ij}_{yy}|  \ll b_{ij}^2,
\qquad |\chi^{ij}_s| \ll b_{ij}^2
\]
The last relation is trivial. For the first two we consider three cases.

\begin{enumerate} 
\item  If $j=0$ and $\e_j < C \tau^{-1}$ then we need no spatial
truncation.  We are allowed to truncate at $|y| > C \tau^{\frac12}$ 
to separate the elliptic region, though.
\item  If $j=0$ and $\e_j > C \tau^{-1}$ 
then 
\[
|\chi^{ij}_y| \lesssim e^{-j(i)}, \qquad |\chi^{ij}_{yy}| \lesssim
e^{-2j(i)}
\]
while
\[
b_{i0}^4 = a_{ij(i)}^4 = \e_{ij(i)} \tau^\frac32 e^{-j(i)} \approx C
e^{-2j(i)} \tau
\]
\item  Otherwise, 
\[
|\chi^{ij}_y| \lesssim e^{-j}, \qquad |\chi^{ij}_{yy}| \lesssim
e^{-2j}
\]
while 
\[
b_{ij}^4 = a_{ij}^4 = \e_{ij} \tau^\frac32 e^{-j} \gtrsim C
e^{-2j(i)} \tau
\]
\end{enumerate}

The results can be summarized by saying that it suffices to construct 
local parametrices in the sets $B_{ij}^k$.

\subsection{Freezing coefficients}

Our first observation is that restricting the result in the
proposition to a single region $B_{ij}^k$ allows us to freeze the
weights $b$, $b_\perp$ in the $X_2$ norms.

Next we are interested in freezing the coefficients of $P$.  We
consider the same three cases as above:

\subsubsection{The case $\e_i \le \tau^{-1}$} 
 In this case we are localized to 
\[
B_{i0} = [i,i+1] \times \R^n
\]
and we have 
\[
b_{i0} \approx 1, \qquad \e_{ij} \lesssim \tau^{-1}
\]
By \eqref{gbound} the second relation  leads to
\[
|d| \lesssim \tau^{-1}
\]
Then using also \eqref{bd} we can estimate the terms involving $d$ in
the expression \eqref{regg} for $P_\psi$,
\begin{equation}
\| \d d \d v\|_{X_2^*} + \tau \|d v\|_{X_2^*} + \tau^\frac12 
\| (d \d + \d d)v\|_{X_2^*} \lesssim \|v\|_{X_2}
\label{eld}\end{equation}
Hence without any restriction of generality we can assume that
$d=0$ in $P_\psi$, which corresponds to taking $g = I_n$.

We also observe that in this case we have 
\[
|\phi| \lesssim 1, \qquad |\phi_y| \lesssim \tau^{-\frac12}
\]
Then we can also drop the $\phi$ component of $\psi$.
Finally, since
\[
|h_{ss}| \lesssim 1
\]
we can replace $h$ by its linearization at some point in the
corresponding $s$ region.

{\em Conclusion: It suffices to prove the result when $d=0$,
  $\psi(y,s) = \tau s$.}

We note that the separation of $\tau$ from integers is no longer
needed due to the localization to unit $s$ intervals.

\subsubsection{The case $\tau^{-1} \le \e_i$}
 In this case we are localized to a region of the form
\[
B_{i0}^k = [s_0,s_0+e^{j(i)} \tau^{-\frac12}] \times B(0,e^{j(i)})
\]
and we have 
\[
b_{i0}^4 \approx \tau e^{-2j(i)}, \qquad \e_{ij} \lesssim e^{-j(i)} \tau^{-\frac12}
\]
The second relation  leads to
\[
|d| \lesssim e^{-j(i)} \tau^{-\frac12}
\]
Then \eqref{eld} is still valid, so we can assume again that $d=0$ in
$P_\psi$.

We also observe that in this case we have 
\[
|\phi| \lesssim 1, \qquad |\phi_y| \lesssim e^{-j(i)}
\]
Then we can also drop the $\phi$ component of $\psi$.
Finally, since
\[
|h_{ss}| \lesssim e^{-j(i)} \tau^{-\frac12}
\]
we can replace $h$ by its linearization at some point in the
corresponding $s$ region.

{\em Conclusion: It suffices to prove the 
result when $d=0$, $\psi(y,s) = \tau s$.}

\subsubsection{ The case $\tau^{-1} \le \e_i$} 
  In this case we are localized to a region of the form
\[
B_{ij}^k = [s_0,s_0+\tau^{-\frac34} e^\frac{j}2 \e_{ij}^{-\frac12}]
\times B(y_0,\tau^{-\frac14} e^\frac{j}2 \e_{ij}^{-\frac12}), \qquad
|y_0| \approx e^{j}
\]
and we have 
\[
b_{ij}^2 \approx \e_{ij}^{\frac12} \tau^{\frac34} e^{-\frac{j}2}  
\]
Using \eqref{gbound} it follows that 
\[
 |g(s,y) - g(s_0,y_0)| \lesssim \e_{ij}^\frac12 \tau^{-\frac14} e^\frac{j}2
\]
Arguing as before, this allows us to freeze $d$ within $B_{ij}^k$.
However, we note that we are no longer allowed to replace 
$d$ by $0$.

Next we turn our attention to the weight function $\psi$.
First we have 
\[
|h_{ss}| \lesssim \e_j \tau \lesssim \e_{ij} \tau^\frac32 e^{-j}
\]
which allows us to replace $h$ by its linearization in $s$ at $s_0$.

Secondly, we claim that we can replace $\phi$ by its linearization
at $y_0$. In the radial direction we have weaker localization
but a stronger bound
\[
|\phi_{rr}| \lesssim \e_{ij} \tau^{\frac12} e^{-j}
\]
 In the transversal direction we have better localization but a weaker
bound,
\[
|\phi_{yy}| \lesssim \e_{i} \tau^{\frac12} e^{-j}.
\]
The first bound allows us to obtain the relation
\[
|\phi_y^2(y,s) - \phi_y^2(y_0,s_0)| \lesssim  \e_{ij}^{-\frac12}
\tau^{\frac34} e^{-\frac{j}2}  \approx b_{ij}^2
\]
Using also the second bound we can write
\[
(\phi_y(y,s) - \phi_y(y_0,s_0)) \partial_y = \nu_r \partial_r + \nu_\perp
\partial_\perp
\]
where the coefficients $\nu_r$ and $\nu_\perp$ are smooth on the
$B_{ij}^k$ scale and satisfy the bounds
\[
|\nu_r| \lesssim  \tau^{-\frac12}  b_{ij}^2, \qquad |\nu_r| \lesssim  
\tau^{-\frac12}  b_{ij,\perp}^2
\]

{\em Conclusion: It suffices to prove the 
result when $d=g(s_0,y_0)$, $\psi(y,s) = \tau s + c y$, $|c| \lesssim
\e_i \tau^{\frac12}$.}

\subsubsection{Additional simplification in the highly localized case}

Given the above simplifications we need to work with a constant
coefficient operator $P_\psi$ which has the form
\begin{eqnarray*}
P_\psi &=& - \d_t + H -\tau + \d d \d +  c \d, \qquad |d| \lesssim
\e_{ij},
\quad |c| \leq \e_i
\end{eqnarray*}
We diagonalize the second order part with  a linear change of variables 
to obtain
\[
P_\psi =- \d_t + H -\tau +  c \d  + O(\e_{ij}) y^2  
\]  
We can freeze the last term at $y_0$ and add it into $\tau$.
To deal with $c$ we make the change of variable 
\[
y \to y- (s-s_0) c
\]
Then our operator becomes
\[
\Ptpsi = - \d_t +   \Delta  -  (y-c(s-s_0))^2 + \tau   
\]  
and the $s-s_0$ terms are negligible due to the $s$ localization.

{\em Conclusion: We can assume without any restriction in generality  that $g=\id$ and $\psi = \tau s$.}

\subsection{The localized parametrix.} 
We begin with the global parametrix $K$ constructed in
Section~\ref{slpflat}. Then we define the parametrix $T_B$ in $B$ by
\[
T_B = \chi_{2B} K, \qquad B=B_{ij}^k
\]
and show that it satisfies \eqref{Tb1} and \eqref{TP1}.

The $L^p$ part of  \eqref{Tb1} follows directly from
\eqref{flatlp2h}. It remains to prove the $X_2$ part,
\[
\|T_B f\|_{X_2} \lesssim \|f\|_{L^1L^2+L^{p'}
  L^{q'}}
\]
The elliptic part of the $X_2$ bound, namely
\[
\|a_e^w(x,D)  \chi_{2B} K f\|_{L^2} \lesssim \|f\|_{L^1L^2+L^{p'}
  L^{q'}},
\]
is obtained by an argument which is similar to the one beginning with
\eqref{ellreg}.

For the rest we consider two cases.

i) If $j = 0$ then $B$ is a ball, and we can use \eqref{flatlp2h}
directly with $R=d$.

ii) If $j > 0$  then $B$ is contained in a sector $B \subset B_{R,d}$
but may be shorter than $R$. This is why we can use 
\eqref{flatlp2h} for the angular part of the $X_2$ norm, but
not for the $L^2$ part. However, the $L^2$ part can be always 
obtained by taking advantage of the time localization,
\[
\| b^{ij} T_B f\|_{L^2} \lesssim \| T_B f\|_{L^\infty L^2} \lesssim
\|f\|_{L^1L^2+L^{p'} L^{q'}}
\]

It remains to consider the error estimate \eqref{TP1}. We have
\[
f - (\d_s-H+\tau)T_B f = [\chi_{2B},\d_s-H+\tau] K f =
 [\chi_{2B},\d_s-H+\tau]\chi_{4B} K f
\]
But arguing as above $\chi_{4B} K_0 f$ satisfies the same
$X_2$ bound as $\chi_{2B} K_0 f$.
Hence it suffices to show that
\[
 [\chi_{2B},\d_s-H+\tau] : X_2 \to X_2^*
\]
This is where the dimensions of the set $B$ are essential; they are
chosen to be minimal so that the above property holds.
We have
\[
\begin{split}
 [\chi_{2B},\d_s-H+\tau] &= -  \d_s \chi_{2B} + (\d_y \chi_{2B}) \d_y + 
\d_y (\d_y \chi_{2B}) 
\\& \hspace{-25pt} =  -  \d_s \chi_{2B} + (\d_r \chi_{2Br}) \d_r + 
\d_r (\d_r \chi_{2Br}) + (\d_\perp\chi_{2By}) \d_\perp + 
\d_\perp (\d_\perp\chi_{2B}) 
\end{split}
\]
For the first factor we use the bound
\[
| \d_s\chi_{2B} | \lesssim b_{ij}^{2}
\]
For the radial derivatives of $\chi_{2B}$ we combine \eqref{bd}
with 
\[
 |\d_r \chi_{2B} | \lesssim \tau^{-\frac12}  b_{ij}^{2}
\]
Finally, for the angular derivatives we use the angular $H^\frac12$
norm in $X_2$ and the bound
\[
 |\d_r \chi_{2B} | \lesssim \tau^{-\frac12}  b_{ij\perp}^{2}
\]

\end{proof}

\section{The gradient term}

In this section we consider the full problem, i.e. involving also the
gradient potential $W$. Ideally one might want to have a stronger
version of Theorem~\ref{mainys} which includes additional bounds for
the gradient, more precisely for
\[
\| e^{\psi(s,y)} \nabla u\|_{L^2}
\]
But such bounds cannot hold, for this would imply that one can improve
the $L^p$ indices in a restriction type theorem. To overcome this
difficulty we proceed as in \cite{MR2001m:35075}, using Wolff's osculation Lemma.
Wolff's idea is that by varying the weight one can ensure
concentration in a sufficiently small set, in which the gradient 
potential term is only as strong as the potential term.
Thus we still obtain a one parameter family of Carleman estimates,
but with the weight depending not only on the parameter but also on 
the function we apply the estimate to. 

Given a gradient potential $W$ satisfying \eqref{W}, we first readjust
the parameters $\e_{ij}$, $\e_i$ constructed in Lemma~\ref{lepsilon}
in order to insure that we have the additional condition
\[
\| W\|_{L^{n+2}(A_i^\tau)} \ll \e_i
\]
Then we begin with the spherically symmetric weights $\psi$ constructed in
Section~\ref{convex} and modify them as follows:
\begin{equation}
\Psi(s,y) = \psi(s,y) + \delta k(s,y)
\end{equation}
where the perturbation $k$ is supported in $\{ |y| \leq 9 \tau \}$
and is subject to the following conditions:
\begin{equation}
|\d_s^\alpha \d_y^\beta \d_\perp^\beta k(s,y)  | \lesssim  
\e_i  \tau^{1 -\frac{\alpha}2} \qquad s \in [i,i+1] 
\label{kbd}\end{equation}
Here $\delta$ is a sufficiently small parameter. 

In order to prove the strong unique continuation result in the 
presence of the gradient potential  $W$ we need
the following modification of Theorem~\eqref{mainys}:

\begin{theorem}
  Assume that \eqref{g} holds. Then for each $\tau > 0$ and 
$W$ subject to  
\[
 \|W\|_{L^{n+2}(A_i^\tau)} \leq \e_i
\]
and each function $u$ vanishing of infinite order at $(0,0)$ and
$\infty$ there exists a perturbation $k$ as in \eqref{kbd} so that
\begin{equation}
  \label{jk1}
 \|e^{\Psi}  u\|_{X} + 
\|e^{\Psi} W \nabla  u\|_{X^*} + \|e^{\Psi}  \nabla (W u)\|_{X^*} + 
\tau^\frac12 \|e^{\Psi} W  u\|_{X^*}
\lesssim \|e^{\Psi(x)}  \tilde P u\|_{X^*} 
\end{equation}
\end{theorem}

Here and in the sequel we will omit indices for $W$. 
After returning to the $(x,t)$ coordinates and taking \eqref{W} into
account this implies Theorem~\ref{carleman}. The reader
should note that the choice of $\phi$ depends on both $u$ and $W$.
This is essential since for fixed $\phi$ (\ref{jk1}) cannot hold
uniformly for all $u$ and $W$.

\begin{proof}
Up to a point the proof follows the steps which were  discussed
in detail before. We outline the main steps:

{\bf STEP 1:} Show that the $L^2$ Carleman estimate \eqref{cpf}
holds with $\psi$ replaced by $\Psi$ {\em for all} perturbations $k$
as in \eqref{kbd}. The new conjugated operator $P_\Psi$ is obtained
from $P_\psi$ after conjugating with respect to the weight
$e^{k(y,s)}$. This adds a few extra components to the selfadjoint 
and skewadjoint parts,
\[
L_\Psi^r = L_\psi^r + k_y^2 + k_s + 2 k_y(\psi_y +d)
\]
\[
L_\Psi^i = L_\psi^i - k_y(1+d) \partial  - \partial k_y(1+d)
\]
Observing that we can write 
\[
k_y^2 + k_s + 2 k_y(\psi_y +d) = \tau d,  \qquad k_y(1+d) =
\tau^\frac12 d
\]
with $d$ as in \eqref{regg} we conclude that the conjugated operator
$P_\Psi$ retains the same form as $P_\psi$, therefore the proof 
of \eqref{cpf} rests unchanged.

{\bf STEP 2:} Show that the symmetric $L^2$ Carleman estimate
\eqref{cpf} holds with $\psi$ replaced by $\Psi$ {\em for all}
perturbations $k$ as in \eqref{kbd}. Since $P_\Psi$ has the same form 
as $P_\psi$, this argument is identical.

{\bf STEP 3:} Show that the symmetric mixed $L^2\cap L^p$ Carleman
estimate in Theorem~\ref{mainys} holds with $\psi$ replaced by $\Psi$
{\em for all} perturbations $k$ as in \eqref{kbd}. Since $P_\Psi$ has
the same form as $P_\psi$, this argument is also identical.

{\bf STEP 4:} Decompose $W$ into a low and a high Hermite-frequency
part, 
\[
W =  W_{low} + W_{high}, \qquad W_{low}= \chi_i^1(x,D) W 
\]
where the smooth symbol $\chi_i^1(x,\xi)$ is supported 
in $\{ x^2 +\xi^2 \leq 81 \tau \}$ and equals $1$ in the region
$ \{ x^2 +\xi^2 \leq 64 \tau \}$. Then we show that the 
high frequency part of $W$ satisfies the desired estimates
 {\em for all} perturbations $k$ as in \eqref{kbd}, namely
\begin{equation}
  \|e^{\Psi} W_{high} \nabla v\|_{X^*} + \|\nabla W_{high} v\|_{X^*}
  +\tau^\frac12\|e^{\Psi} W_{high} v\|_{X^*} \lesssim
 \| e^{\Psi} u\|_{X} \|W\|_{L^{n+2}}
\label{Whigh}\end{equation}
After conjugation this becomes
\[
\| W_{high} \nabla v\|_{X^*} + \|\nabla W_{high} v\|_{X^*}+
\tau^\frac12\| W_{high} v\|_{X^*} \lesssim \| v\|_{X} \|W\|_{L^{n+2}}
\]
We only consider the first term on the left. The second one is
equivalent by duality, and the third one is similar but simpler.  We
divide $v$ into two components,
\[
v = (1-\chi_e^1) v + \chi_e^1 v
\]
where the smooth symbol $\chi_e^1(x,\xi)$ is supported 
in $\{ x^2 +\xi^2 \geq 9 \tau \}$ and equals $1$ in the region
$ \{ x^2 +\xi^2 \geq 10  \tau \}$. 

For the high frequency component of $v$ we use the $H^1$ part 
of the $X^2$ norm to estimate 
\[
\| W_{high} \nabla  \chi_e^1 v\|_{L^{\frac{2(n+2)}{n}}} \lesssim
\|W_{high}\|_{L^{n+2}} \|\nabla  \chi_e^1 v\|_{L^2} \lesssim 
\|W\|_{L^{n+2}} \|v\|_{X}
\]
For the low frequency component of $v$ it is still possible to
estimate directly the high frequency of the output,
\[
\begin{split}
\| \chi_e^1 (W_{high} \nabla  (1-\chi_e^1) v) \|_{H^{-1}} \lesssim&
\tau^{-\frac12}
\| W_{high} \nabla  (1-\chi_e^1) v\|_{L^2} 
\\ \lesssim& \|W_{high}\|_{L^{n+2}}  \tau^{-\frac12}  \| \nabla  (1-\chi_e^1)
v\|_{L^{\frac{2(n+2)}n}} \\ \lesssim & \|W\|_{L^{n+2}}   \|
v\|_{L^{\frac{2(n+2)}n}} 
\end{split}
\]

Finally, the last remaining part has a much better $L^2$ estimate,
\[
\|(1-\chi_e^1) (W_{high} \nabla  (1-\chi_e^1) v)\|  \lesssim 
\tau^{-N} \|W\|_{L^{n+2}}    \| v\|
\]
which is due to the unbalanced frequency localizations of the two
factors.

Due to the estimate \eqref{Whigh}, it suffices to prove \eqref{jk1} with
 $W$ replaced by $W_{low}$. This allows us to replace the 
term $\nabla (W_{low} v)$ by
\[
\nabla (W_{low} v) = W_{low} \nabla v +  (\nabla W_{low})  v 
 \]
where we can estimate 
\[
\| \nabla W_{low} \|_{L^{n+2}} \lesssim \tau^\frac12 \|W\|_{L^{n+2}}
\]
Hence without any restriction in generality we can drop the third 
term in \eqref{jk1} and show that we can choose the perturbation $k$
so that 
\begin{equation}
  \label{jk2}
  \|e^{\Psi} W \nabla  u\|_{X^*}  + 
  \tau^\frac12 \|e^{\Psi} W  u\|_{X^*}
  \lesssim \|e^{\Psi(x)}  \tilde P u\|_{X^*} 
\end{equation}

{\bf STEP 5:} Show that, given $u$ and $W$, {\em we can choose } the
perturbation $k$ so that \eqref{jk2} holds. At this stage we no longer
need the full $X^*$ norm for the $W$ terms, it suffices instead to
consider the $L^\frac{2(n+2)}{n+4}$ norm. Begin with the unperturbed
integral
\[
\int_\R \int_\R^n F_\psi dx dt, \qquad F_\psi = |e^\psi W \nabla
u|^\frac{2(n+2)}{n+4} + |\tau^\frac12 e^\psi W
u|^\frac{2(n+2)}{n+4}
\]
We can select a subset $I$ of $\R$ consisting of time intervals of
length $1$ with unit separation at least $8$ so that
\[
\int_\R \int_\R^n  F_\psi dx dt \lesssim \int_I \int_\R^n  F_\psi dx dt
\]

By a small abuse of notation we label 
\[
I = \bigcup_{i \in \mathcal I}  I_i, \qquad  I_i \subset [i-1,i+1]
\]
We define a family of perturbations $k$ depending on parameters $b_i$,
$\sigma_i$ by
\[
\begin{split} 
k(y,s) = &  \sum_{i \in \mathcal I} \e_i \left(
100 \chi_{3 \tau I_{i}}  + \chi_{2 I_{i}} \chi_{|y|^2\le \tau}  (b_i y
+ \sigma_i(s-i))\right),
\\ &  \qquad |b_i| \leq \tau^\frac12, \ \ |\sigma_i| \leq \tau
\end{split} 
\]
Due to the choice of the intervals $I_i$ it is easy to see that after
changing the weight $\psi$ to $\Psi$ we retain the concentration
to a dilate of $I$,
\[
 \int_\R \int_{\R^n}  F_\Psi dx dt \lesssim \int_{3I} \int_{\R^n}  F_\Psi dx dt
\] 

The choice of the parameters $b_i$, $\sigma_i$ can be  made 
independently for each $i$. We consider two cases.

i) Suppose $\e_i \lesssim \tau^{-\frac12}$. Then the choice of the
parameters is irrelevant since in $3I_i$ we can estimate
\[
\begin{split}
\| e^\psi W \nabla u\|_{L^\frac{2(n+2)}{n+4}} 
+ \tau^\frac12 \| e^\psi W u\|_{L^\frac{2(n+2)}{n+4}} 
& \lesssim \|W\|_{L^{n+2}} ( \| e^\Psi \nabla u\|_{L^2} 
+ \tau^\frac12 \| e^\Psi  u\|_{L^2} )
\\ &  \lesssim ( \tau^{-1/2} \| e^\Psi \nabla u\|_{L^2} 
+  \| e^\Psi  u\|_{L^2} )
\\ &  \lesssim \|u\|_{X_2}
\end{split}
\]
 
ii)  Suppose $\e_i \gg \tau^{-\frac12}$. Then we need to choose 
the parameters $b_i$, $\sigma_i$ in a favorable manner. 
This choice is made using Wolff's Lemma:

\begin{lemma}[Wolff's Lemma~\cite{780.35015}]
Let $\mu$ be a measure in $\R^n$ and $B$ a convex set. Then one
can find $b_k \in B$ and disjoint convex sets $E_k \subset \R^n$
so that the measures $e^{xb_k}\mu$ are concentrated in $E_k$,
\[
\int_{E_k}  e^{xb_k} d\mu \gtrsim \frac12 \int_{E_k}  e^{xb_k} d\mu 
\]
and
\[
\sum |E_k|^{-1} \gtrsim |B|
\]
\end{lemma}

We apply the lemma for the measures 
\[
d \mu_i = 1_{3I_i} F_\psi
\]
In our case we have 
\[
B_i = \delta \e_i \left( [-\tau,\tau]\times B(0,\tau^\frac12)\right),
\qquad 
|B_i| \approx \e_i^{n+1} \tau^\frac{n+2}{2}
\]
Hence we can find parameters $b_i^k$ and $\sigma_i^k$ and convex 
sets $E_i^k \subset 3I_i \times B(0,3\tau^\frac12)$  so that 
the corresponding measures $F_{\Psi_k}$ are concentrated in $E_i^k$
with
\[
\sum |E_i^k|^{-1} \gtrsim  \e_i^{n+1} \tau^{\frac{n+2}2}
\]
At the same time we have
\[
\sum \int_{E_i^k} |W|^{n+2} dxdt \lesssim \e_i^{n+2}
\]
Hence we can choose $k$ so that 
\[
\int_{E_i^k} |W|^{n+2} dxdt \lesssim \e_i\tau^{-\frac{n+2}2} 
|E_i^k|^{-1}
\]
which by Holder's inequality leads to
\begin{equation}
\| W\|_{L^{\frac{n+2}2} (E_i^k) } \lesssim \e_i^\frac{1}{n+2} \tau^{-\frac12}
\label{esmall}\end{equation}
Denoting this index $k$ by $k(i)$  we can write
\[
\begin{split}
  \| e^\Psi W (\nabla,\tau^\frac12) u\|_{l^2 L^\frac{2(n+2)}{n+4}}
  &\lesssim \| e^\Psi W (\nabla,\tau^\frac12) u\|_{l^2_i
    L^\frac{2(n+2)}{n+4}(I_i)} 
\\ &\lesssim \| e^\Psi W
  (\nabla,\tau^\frac12) u\|_{l^2_i L^\frac{2(n+2)}{n+4}(E_i^{k(i)})} 
\\ &\lesssim \|  W
  (\nabla,\tau^\frac12) (e^\Psi u)\|_{l^2_i L^\frac{2(n+2)}{n+4}(E_i^{k(i)})} 
\end{split}
\]
Decomposing the function $v = e^\Psi u$ into low and high frequencies
we further estimate
\[
\begin{split}
  \| e^\Psi W (\nabla,\tau^\frac12) u\|_{l^2 L^\frac{2(n+2)}{n+4}}
  \lesssim & \
\|  W
  (\nabla,\tau^\frac12)  (1- \chi_{i}^1(x,D))  (e^\Psi u)\|_{l^2_i
    L^\frac{2(n+2)}{n+4}(I_i)} 
 \\ & + 
  \|  W
  (\nabla,\tau^\frac12) \chi_{i}^1(x,D)  (e^\Psi u)\|_{l^2_i
    L^\frac{2(n+2)}{n+4}(E_i^{k(i)})}
\end{split}
\]
The first term on the right is estimated as in Step 4,
\[
\begin{split}
\|  W(\nabla,\tau^\frac12)  (1- \chi_{i}^1(x,D)) &  (e^\Psi u)\|_{l^2_i
    L^\frac{2(n+2)}{n+4}(I_i)} 
\\ \lesssim &\ \| W\|_{l^\infty L^{n+2}} \|  (1- \chi_{i}^1(x,D)) (e^\Psi
u)\|_{H^1} \\ \lesssim &\ \| W\|_{l^\infty L^{n+2}} \| e^\Psi
u\|_{X}
\end{split}
\]
It is only for the second term on the right that we need to use
\eqref{esmall}:
\[
\begin{split}
&  \|  W
  (\nabla,\tau^\frac12)  \chi_{i}^1(x,D)  (e^\Psi u)\|_{l^2_i
    L^\frac{2(n+2)}{n+4}( E_i^{k(i)})} 
 \\ &\lesssim \|W\|_{l^\infty_i L^{\frac{n+2}2}(E_i^{k(i)})} \| 
  (\nabla,\tau^\frac12) \chi_i^1(x,D) (e^\Psi u)\|_{l^2_i L^\frac{2(n+2)}{n+4}(E_i^{k(i)})}
\\ &\lesssim  \| e^\Psi  u\|_{l^2 L^\frac{2(n+2)}n} 
\\ &\lesssim  \| e^\Psi  u\|_X
\end{split}
\]
The proof of the Theorem is concluded.
\end{proof}

\appendix

\section{The change of coordinates}
\label{coordinates}
Suppose that the coefficients $g$ satisfy \eqref{g}. In this section
we verify that we can change coordinates so that \eqref{g} and
\eqref{g1} are both satisfied.  Due to the anisotropic character of
the equation we must leave the time variable unchanged and consider
changes  of coordinates which have the form
\[ 
(t,x) \to (s,y), \qquad s=t, y = \chi(t,x). 
\]
The expression for the operator $P$ in the new coordinates
is
\[
P = \d_t + \d_k \tilde{g}^{kl}(t,y) \d_l + \frac{x_k}t \tilde{d}^{kl} \d_l 
\]
where the new coefficients $\tilde{g}$, $\tilde d$ are computed using
the chain rule,
\[
\tilde g^{kl} = \frac{\d \chi_l}{\d x_i}  g^{ij}  \frac{\d \chi_k}{\d
  x_j}, \qquad  
\frac{x_k}t \tilde{d}^{kl} = \frac{\d \chi_l}{\d t} 
 - \left[ D \chi^{-1}_{lm}   (\d^2_{x_m x_i} \chi_l) \right] 
 g^{ij} \frac{d\chi_l}{\d x_j}
\]
There is a price to pay for this, namely in the new coordinates we
obtain lower order terms which cannot be treated perturbatively.
Instead we obtain coefficients $\tilde d^k$ which have the same
regularity and size as $g_1-\id$.

The Lipschitz condition \eqref{g} ensures that $g$ has a limit at
$(0,0)$ so we assume that $g$ is continuous.  After a linear change of
coordinates we may and do choose $g$ with $g(0,0)=\id$. Again by
\eqref{g} this implies
\begin{equation} \label{perg} 
 |g(t,x) -\id | \ll 1.  
\end{equation}

\begin{proposition}\label{coordi} 
  Let $g$ be a metric which satisfies \eqref{g} with $g(0,0)=\id$.
  Then there is change of coordinates $(t,y) = (t,  \chi(t,x))$ 
which is close to the identity
\begin{equation} \label{dev} 
 \Vert  \d_x \chi - \id \Vert_{L^\infty} \ll 1
\end{equation} 
and has regularity
\begin{equation} 
  \sup_\tau 
  \Vert (t+x^2)^{-1/2} (t \d_t)^\alpha ((t+x^2)^{1/2} \d_x)^\beta
  \chi \Vert_{ l^1(A(\tau);L^\infty)} \ll 1, \quad 2 \leq 2\alpha + |\beta|
  \leq 4
\label{chireg}\end{equation}
so that in the new coordinates both functions $\tilde g$ and $\tilde d$
satisfy \eqref{g}, while $\tilde g -\id$ and $\tilde d$ satisfy
\eqref{g1}.
\end{proposition}

\begin{proof} 
Consider the covering of the $[0,2] \times B(0,2) = \cup A_{ij} $ with
an associated smooth partition of unity $\eta_{ij}$. We can assume 
that the functions $\eta_{ij}$ satisfy
\begin{equation}
|\partial_{t}^\alpha \partial_x^\beta \eta_{ij}| \lesssim c_{\alpha
  \beta}
t^{-\alpha} (t+x^2)^{-\frac{\beta}2}
\label{etareg}\end{equation}

We choose the points 
\[ 
(t_i,x_{ij})= (e^{-4i},e^{-2i+j}) \in A_{ij}.
\]
and insure that $\eta_{ij} = 1$ near $(t_i,x_{ij})$. By \eqref{g} we have
\begin{equation} \label{diffg} 
\sup_\tau \sum_{(i,j) \in \mathcal{A}(\tau)} 
| g(t_i , x_{ij})-g(t_i,x_{i,(j-1)})| + |g(t_i,x_{ij})-g(t_{i+1},x_{(i+1),j})|
\ll 1. 
\end{equation}

Within a fixed set $A_{ij}$ we consider the linear map  defined 
by the matrix 
\[ 
\chi_{ij} = g^{-1/2}(t_i, x_{ij}). 
\] 
It transforms the coefficients at $(t_i,x_{ij})$ to the identity and
has the desired properties within $A_{ij}$. We assemble the maps
defined by $\chi_{ij}$ using the partition of unity,
\[ \chi(t,x) = \sum \eta_{ij}(t,x) \chi_{ij} x. \] 
Then 
\[ \nabla \chi(t,x)- \id = \sum (\nabla \eta_{ij}) \chi_{ij} x +
\eta_{ij} (\chi_{ij}-\id), 
\] 
Let $(t,x)\in A_{i_0,j_0}$. Since $\sum
\nabla \eta_{ij}=0$ we have
\[ 
\nabla_x \chi (t,x)-\id  
= \sum_{ \eta_{ij}(t,x) > 0} 
\nabla_x \eta_{ij}(t,x)  ( \chi_{ij} -\chi_{i_0,j_0}) 
+ \eta_{ij} (\chi_{ij}-\id). 
\] 
The first term on the right hand is small by \eqref{diffg} (for
$\chi_{ij}$) and the second one by \eqref{perg} therefore the smallness of
$\nabla \chi -\id$ follows.

For the second order spatial derivatives we write
\[
\begin{split}  D^2_x \chi(t,x)  = &  \sum D^2 \eta_{ij}(t,x)  \chi_{ij} x + 
2 D_x \eta_{ij}(t,x)  \chi_{ij} 
\\ = & 
\sum D_x^2 \eta_{ij}(t,x) (\chi_{ij} - \chi_{i_0,j_0})x 
+ 2 D_x \eta_{ij} (\chi_{ij} - \chi_{i_0,j_0}). 
\end{split} 
\]
Hence by \eqref{etareg} and \eqref{diffg} (again for $\chi_{ij}$) we obtain 
\[ \Vert (|x|^2+t)^{1/2} D_x^2 \chi \Vert_{l^1(A(\tau); L^\infty)} \ll 1. \]
Also
\[ 
\d_t  \chi (t,x)  
= \sum \d_t \eta_{ij} ( \chi_{ij} - \chi_{i_0j_0})x 
\]
gives the desired bound for the time derivative. A similar computation
yields the bound for the higher order derivatives in \eqref{chireg}.

Consider now the new metric $\tilde g$. Since both $D \chi$ and $(D
\chi)^{-1}$ are Lipschitz on the dyadic scale with $l^1(A(\tau))$
summability, from \eqref{g} for $g$ we easily obtain \eqref{g} for
$\tilde g$. In addition, our construction insures that
$\tilde g (t_i,x_{ij}) =\id$. This in turn leads to the bound
\[
\| \tilde g-I_n\|_{L^\infty(A_{ij})} \lesssim  \Vert \tilde g \Vert_{\lip_x(A_{ij})}
+ \|\tilde g\|_{C^{m_{ij}}_t(A_{ij})}
\]
which shows that \eqref{g} for $\tilde g$ implies \eqref{g1}
for $\tilde g-\id$.

It remains to consider the lower order terms. From $\d_t \chi$ we
obtain coefficients $\tilde d$ of the form
\[
\tilde d = t \sum \partial_t \eta_{ij} \chi_{ij} 
\]
Within $A_{i_0,j_0}$ this gives
\[
\tilde d = t \sum \partial_t \eta_{ij}(t,x) (\chi_{ij} -\chi_{i_0 j_0} )
\]
The functions $ t \partial_t \eta_{ij}(t,x)$ are bounded and smooth on
the dyadic scale, while the $l^1(A(\tau))$ summability comes from the
$\chi_{ij} -\chi_{i_0 j_0} $ factor due to \eqref{diffg}. Hence both 
\eqref{g} and \eqref{g1} are satisfied.

The contribution of $\partial_x^2 \chi$ to $\tilde d$ has the form
\[
\tilde d  = \frac{x t}{x^2} (\d_x \chi)^{-1} (\d_x^2 \chi)  g
\]
There is no singularity at $x=0$ since $\chi$ is linear in $x$ for
$x^2 \ll t$. Then from \eqref{dev} and \eqref{chireg} we obtain
\[
|\tilde d| \lesssim \frac{t}{x^2}
\]
with added $l^1(A(\tau))$ summability inherited from $\d_x^2 \chi$.
This is better than \eqref{g1}, and in effect this term can be
included in $W$ and treated perturbatively. 
The bound \eqref{g1} is also easy to obtain from the similar bounds 
for $g$ and derivatives of $\chi$.

\end{proof}



\end{document}